\documentclass[reqno,12pt,centertags]{amsart}
\usepackage{amsmath,amsthm,amscd,amssymb,latexsym,upref,stmaryrd}
\usepackage[cp1251]{inputenc}
\usepackage[english]{babel}
\usepackage{mathtext}
\usepackage{amsfonts}
\usepackage{graphicx}
\usepackage[numbers,sort&compress]{natbib}
\usepackage{hyperref}
\newcommand*{\mailto}[1]{\href{mailto:#1}{\nolinkurl{#1}}}
\usepackage{pgfplots}
\usepackage{amsmath}
\usepackage{amsthm}
\newtheorem{corollary}{Corollary}
\newtheorem{theorem}{Theorem}
\newtheorem{lemma}{Lemma}
\newtheorem{remark}{Remark}
\theoremstyle{proposition}
\newtheorem{proposition}{Proposition}
\newtheorem{ip}{Inverse Problem}
\newtheorem{algorithm}{Algorithm}
\numberwithin{equation}{section}
\begin{document}
\thispagestyle{empty}

\noindent{\large\bf On the local solvability and stability of the partial inverse problems
for the non-self-adjoint Sturm-Liouville operators with a discontinuity}
\\[0.5cm]
\noindent {\bf  Xiao-Chuan Xu}\footnote{School of Mathematics and Statistics, Nanjing University of Information Science and Technology, Nanjing, 210044, Jiangsu,
People's Republic of China, {\it Email:
xcxu@nuist.edu.cn}}
{\bf, Chuan-Fu Yang}\footnote{School of Mathematics and Statistics, Nanjing University of Science and Technology, Nanjing, 210094, Jiangsu,
People's Republic of China, {\it Email: chuanfuyang@njust.edu.cn}}
{ \bf and  Natalia Pavlovna Bondarenko}\footnote{S.M. Nikolskii Mathematical Institute, Peoples' Friendship University of Russia (RUDN University), 6 Miklukho-Maklaya Street, Moscow, 117198, Russian, Email: {\it bondarenkonp@sgu.ru}}\footnote{ Moscow Center of Fundamental and Applied Mathematics, Lomonosov Moscow State University, Moscow 119991, Russia}
\\

\noindent{\bf Abstract.}
{In this work, we study the inverse spectral problems for the  Sturm-Liouville operators on $[0,1]$ with complex coefficients and a discontinuity at $x=a\in(0,1)$. Assume that the potential on $(a,1)$ and some parameters in the discontinuity and boundary conditions  are given. We recover the potential on $(0,a)$ and the other parameters from the eigenvalues. This is the so-called partial inverse problem. The local solvability and stability of the partial inverse problems are obtained  for $a\in(0,1)$,  in which the error caused by the given partial potential is considered. As a by-product, we also obtain two new uniqueness theorems for the partial inverse problem.

}

\medskip
\noindent {\it Keywords:} Non-self-adjoint Sturm-Liouville operator, inverse spectral problem,  discontinuity, local solvability, stability

\medskip
\noindent {\it 2020 Mathematics Subject Classification:} 34A55; 34B05;  34L40; 47E05
\section{introduction}

Consider the following Sturm-Liouville problem
\begin{equation}\label{1}
-y^{\prime \prime}(x)+q(x)y(x)=\lambda y(x), \quad x\in(0,a)\cup (a,1),
\end{equation}
with the boundary conditions
\begin{equation}\label{2}
y'(0)-hy(0)=0,\quad y'(1)+Hy(1)=0,
\end{equation}
and the jump conditions
\begin{equation}\label{3}
y(a+0)=a_1y(a-0),\quad y'(a+0)=a_1^{-1}y'(a-0)+a_2y(a-0),
\end{equation}
where $\lambda $ is spectral parameter, the complex-valued potential $q$ belongs to $L^2(0,1)$, $h,H,a_2\in \mathbb{C}$, $a_1>0$ and $a\in (0,1)$.

Inverse spectral  problems for the Sturm-Liouville operators consist in recovering the coefficients of the operators from their spectral characteristics. The basic results of inverse  Sturm-Liouville problems  can be found in, e.g.,  the monographs \cite{FYU1,PT,VM,Lev}. The Sturm-Liouville problems with discontinuities inside the interval arise in mathematics, mechanics, radio electronics, geophysics, and other fields of science and technology. Such problems are usually connected with discontinuous material properties, for example, the transmission eigenvalue problem with a discontinuous index of refraction \cite{GP}, the  geophysical models for oscillations of the
Earth \cite{AND,HAL} and the electromagnetic and elastic inverse problems  for media with
discontinuous material properties \cite{KRU}.

The Sturm-Liouville problem \eqref{1}--\eqref{3} has attracted much attention of scholars (see, e.g., \cite{AM,FYU2,HAL,SAT,SYU,W,CWI,XY,Y0,YB,YUR1,ZBY} and the references therein). In order to uniquely recover the potential on $[0,1]$ and all the coefficients, one needs to know two spectra \cite{YUR1,SAT,ZBY}. However, when partial information on the potential and a part of coefficients are known a priori, then only a part of two spectra are needed (see, e.g., \cite{HAL,CWI,SYU,XY,Y0}). In particular, roughly speaking, if $q(x)$ is known a priori on a half interval, then only one spectrum is sufficient; if $q(x)$ is given on a subinterval more than one half, then only a part of one spectrum is enough. An overview of classical and modern results on partial inverse Sturm-Liouville problems is presented in \cite{Bond23}.
Without the jump conditions, the uniqueness for solutions of partial inverse problems was considered in \cite{dGS,GS2,MH,MH1,HL} and other works.

In this paper, we consider the  local solvability and stability the partial inverse problems for the problem \eqref{1}--\eqref{3} with complex coefficients, in which, generally  speaking,  a part of coefficients and $q(x)$ on $(a,1)$ are known a priori. It is known \cite{YB} that  the problem \eqref{1}--\eqref{3} is equivalent to the following problem $B_1=B(d_1,d_2,q_1,q_2,h,H ,a_1,a_2):$
\begin{equation}\label{1s}
{-y_j^{\prime \prime}(x)+q_j(x)y_j(x)=\lambda y_j(x), \quad 0<x<d_j},\quad j=1,2,
\end{equation}
\begin{equation}\label{2s}
y_1'(0)-hy_1(0)=0,\quad y_2'(0)-Hy_2(0)=0,
\end{equation}
\begin{equation}\label{3s}
y_1(d_1)-a_1^{-1}y_2(d_2)=0,\quad
y_1'(d_1)+ [a_1y_2'(d_2)+a_2y_2(d_2)]=0,
\end{equation}
where $d_1=a$, $d_2=1-a$, $q_1(x)=q(x)$ for $x\in(0,d_1)$ and $q_2(x)=q(1-x) $ for $x\in (0,d_2)$.

Since the coefficients are complex, the problem \eqref{1}--\eqref{3} is  non-self-adjoint. Thus, there may exist multiple and non-real eigenvalues. The asymptotic behavior of the eigenvalues is the same as that in the self-adjoint case. So, the appearance of non-real eigenvalues cause almost no difficulty in studying the inverse problems. However, the appearance of multiple eigenvalues will cause the main difficulties in the non-self-adjoint cases when we  study
the inverse problems, especially, for the local solvability and stability (see, e.g.,  \cite{BB,BN1,BN2,B,BK,HK,MW}). In order to overcome the difficulties caused by multiple eigenvalues, we shall develop the methods and techniques of \cite{BB,BN1,BN2,BK,HK,MW}.
Let  $\{\lambda_{1,n}\}_{n\ge0}$ (counted with multiplicities) be the eigenvalues of the problem $B_1$.
\begin{ip}\label{ip1}
 Assume that $a\in(0,1/2]$, and let $I$ be the subset of $\mathbb{N}_0:=\mathbb{N}\cup\{0\}$. Given $\{\lambda_{1,n}\}_{n\in I}$, $a_1$, $a_2$, $H$, $\omega_1:=h+\frac{1}{2}\int_0^aq_1(x)dx$ and $q_2$, find $q_1$ and $h$.
\end{ip}
In the self-adjoint case, the local solvability and stability of Inverse Problem \ref{ip1} are proved  in \cite{YB}, where, in particular, if $a=1/2$ then $I=\mathbb{N}_0$ and the coefficients $a_2,\omega_1$ can be recovered from the spectrum.
 In this paper, we consider the local solvability and stability  for the non-self-adjoint case. Moreover, we shall also study the case $a\in(1/2,1)$.
It is known that if  $a\in(1/2,1)$ and $q(x)$ on $(a,1)$ is given, then one spectrum is not sufficient to uniquely determine the potential on the whole interval. One should add some other eigenvalues. Consider the problem $B_0=B(d_1,d_2,q_1,q_2,h,\infty ,a_1,a_2)$ which  means that $y'(1)+Hy(1)=0$ is replaced by $y(1)=0$.
Let $\{\lambda_{0,n}\}_{n\ge0}$ (counted with multiplicities) be the eigenvalues of the problem $B_0$.
Let us  consider the following Inverse Problem \ref{ip2}.

\begin{ip}\label{ip2}
Let $I_i$ ($i=0,1$) the subsets of $\mathbb{N}_0$. Given  $\{\lambda_{i,n}\}_{n\in I_i}\ (i=0,1)$, $a_1,$ $a_2$, $H$, $\omega_1:=h+\frac{1}{2}\int_0^aq_1(x)dx$ and $q_2$, find $q_1$ and $h$.
\end{ip}

Note that, since $a\in(0,1)$ in Inverse Problem \ref{ip2}, it includes Inverse Problem \ref{ip1} as a special case. For example, if $a\in(0,1/2]$, then we can put $I_1=I$ and $I_0=\emptyset$. The local solvability of  Inverse Problems \ref{ip2} depends on the index sets  $I_i$ ($i=0,1$). We will give the descriptions of these index sets in the corresponding main theorems. For convenience of formulation, we shall renumber the sequences $\{\lambda_{i,n}\}_{n\in I_i}$ by letting $\{\mu_{i,k}\}_{k\ge0}=\{\lambda_{i,n}\}_{n\in I_i}$ with $|\mu_{i,k+1}|\ge |\mu_{i,k}|$, $i=0,1$. Our results can also be generalized into the partial inverse problems from parts of $N+1$ spectra of the problems with different boundary conditions at $x=1$, where $N\ge1$. But the proportion of the needed eigenvalues should remain the same (see Remark \ref{remark2}).

Another motivation of this paper is that, in the local solvability and stability, the error caused by the given potential $q_2$ and the parameter $H$ should be  considered. In \cite{YB}, the authors gave the local solvability and stability for Inverse Problem \ref{ip1} by assuming that only the given subspectrum contains $\varepsilon$-error and the given potential $q_2$ and the parameter $H$ contain no error. However, as one of the input data, the given potential $q_2$ or the parameter $H$ may also contain $\varepsilon$-error. Therefore, for Inverse Problem \ref{ip1}, it is natural to consider the local solvability and stability with the subspectrum, the potential $q_2$ and $H$ containing $\varepsilon$-error.
For Inverse Problem \ref{ip2}, since partial eigenvalues of the problem $B_0$ are a part of the input data,  the given parameter $H$ is assumed to contain no error, and the subspectra and the potential $q_2$ contain $\varepsilon$-error.

 In addition to \cite{YB},  let us also discuss the essential novelties of our results comparing with some other previous studies. First, consider the case of real-valued potential $q_1(x)$ and simple eigenvalues $\{ \lambda_{i,n}\}$. In this special case, Inverse Problem \ref{ip2} can be treated as the problem of Horv\'{a}th \cite{MH} (see Remark~\ref{remnew}), which consists in recovering $q_1$ and $h$ from the eigenvalues $\{\lambda_n\}_{n\ge1}$
of the following problems
\begin{equation} \label{Horvath}
-y'' + q_1(x) y = \lambda y, \quad y'(0) - hy(0) = 0,\quad  y(d_1)\cos\alpha_n + y'(d_1)\sin\alpha_n = 0,\quad n \ge 1.
\end{equation}
We mean that the eigenvalues are taken from different spectra: $\lambda_n \in \sigma(q_1, h, \alpha_n)$, $n \ge 1$, where $\sigma(q_1, h, \alpha_n)$ is the spectrum of \eqref{Horvath}.
In the case of $h=\infty$ (i.e. the Dirichlet boundary condition at $x=0$), Horv\'{a}th \cite{MH} gave a necessary and sufficient condition for the uniqueness of the inverse problem solution. In the case of $h\in \mathbb{R}$, Horv\'ath has separately obtained a necessary condition and a sufficient one for the unique determination of $q_1$ and $h$. The necessary and sufficient conditions of \cite{MH} are formulated in terms of the closedness for some exponential systems.
Latter on, in the case of $h=\infty$, Horv\'{a}th and Kiss studied the stability for the self-adjoint case in \cite{HK0} and for the non-self-adjoint case in \cite{HK}.
However, in this paper, we investigate the local solvability of the inverse problem, which was not considered by Horv\'ath and Kiss. Moreover, in the case of multiple eigenvalues, our problem statement is different from \cite{MH, HK0, HK}, so it requires a separate investigation. Also, we obtain a necessary and sufficient condition for the uniqueness of solution for Inverse Problem \ref{ip2} in terms of the completeness for a sequence of two-element vector functions. This condition is different from that in \cite{MH}.

Second, the boundary value problem \eqref{1s}--\eqref{3s} can be represented in the following form:
\begin{equation} \label{1e}
-y''(x) + q_1(x) y(x) = \lambda y(x), \quad 0 < x < d_1,
\end{equation}
\begin{equation} \label{2e}
y'(0) - h y(0) = 0, \quad f_1(\lambda) y'(d_1) + f_2(\lambda) y(d_1) = 0,
\end{equation}
where $f_1(\lambda)$ and $f_2(\lambda)$ are some entire  functions, which are constructed by the known data $q_2(x)$, $H$, $a_1$, $a_2$ (see Remark~\ref{remnew} for details). Thus, Inverse Problem \ref{ip1} is reduced to the recovery of $q_1$ and $h$ from a subspectrum of \eqref{1e}--\eqref{2e}, while the entire functions $f_1(\lambda)$ and $f_2(\lambda)$ are known a priori.  However, Inverse Problem~\ref{ip2} cannot be represented as \eqref{1e}--\eqref{2e}, since it implies different functions $f_1^i(\lambda)$ and $f_2^i(\lambda)$ for the problems $B_i$ with $i = 0$ and $i = 1$.
For the problems of the form  \eqref{1e}--\eqref{2e}, the inverse spectral theory has been created in \cite{BN1,BN2,YBX} and subsequent studies (see the overview \cite{Bond23}). In particular, the case of multiple eigenvalues was considered and local solvability and stability of inverse problems were proved, whereas, when $f_1(\lambda)$ and $f_2(\lambda)$ are fixed. In the current paper, $f_1(\lambda)$ and $f_2(\lambda)$ depend on the given potential  $q_2$ and some parameters in the boundary and jump conditions. As a result, $f_1(\lambda)$ and $f_2(\lambda)$ may not be fixed, since, as mentioned before, the given  $q_2$ and $H$ may contain errors. In this paper, we shall take this fact into account.

It should be mentioned that, for the case of the Dirichlet boundary condition without discontinuity ($a_1 = 1$, $a_2 = 0$, $h = H = \infty$, $a = 1/2$), the local solvability and stability of the inverse problem with  the perturbation of the known potential for the first time was proved in \cite{BB}, during the investigation of the inverse transmission eigenvalue problem.  Nevertheless, the presence of the discontinuity causes additional difficulties and  so requires a separate investigation. Moreover,  the methods used here, especially in dealing with the perturbation of the known potential,  are different from that in \cite{BB}. 

The paper is organized as follows. In Section 2, we derive the main equations and introduce the main results in this paper, including two uniqueness theorems, a reconstruction algorithm, a theorem of the local solvability and stability, and two corollaries. In Section 3, we study the asymptotic behavior of the vector functional sequence in the main equations.
In Section 4, we provide the proofs of the uniqueness theorems for Inverse Problem \ref{ip2}. In Section 5, we prove the theorem for the local solvability and stability of Inverse Problems \ref{ip1} and \ref{ip2}. In particular, we first consider the general case $a\in(0,1)$, namely, Inverse Problem \ref{ip2}, in which the given $q_2$ contains $\varepsilon$-error. Then, we consider the case $a\in(0,1/2]$, namely, Inverse Problem \ref{ip1}, in which the given $H$ and $q_2$ contain $\varepsilon$-error. In Appendix, some auxiliary propositions of complex and functional analyses are provided.

\section{Main results}
In this section, we first derive the main equations for solving Inverse Problems~\ref{ip1} and~\ref{ip2}, and then present the main results of this paper.

Let $\varphi(x,\lambda)$ be the solution of \eqref{1s} for $j=1$ satisfying the initial conditions $\varphi(0,\lambda)=1,\varphi'(0,\lambda)=h$. Let $\psi_0(x,\lambda)$ and $\psi_1(x,\lambda)$ be the solutions of \eqref{1s}  for $j=2$  satisfying the initial conditions
\begin{equation} \label{initpsi}
 \psi_0(0,\lambda)=0,\; \psi_0'(0,\lambda)=1,\quad \psi_1(0,\lambda)=1,\;\psi_1'(0,\lambda)=H,
\end{equation}
respectively.
Then, in view of~\eqref{3s}, the eigenvalues $\{\lambda_{i,n}\}_{n\ge0}$ of the problem $B_i$ ($i=0,1$), respectively, coincide with the zeros of the   characteristic functions
\begin{equation}\label{2.0}
 \Delta_i(\lambda)=\left|
                   \begin{array}{cc}
                    \varphi(a,\lambda) & -a_1^{-1}\psi_i(d_2,\lambda)\\
                    \varphi'(a,\lambda)& a_1\psi_i'(d_2,\lambda) +a_2\psi_i(d_2,\lambda) \\
                   \end{array}
                 \right|:=\left|
                   \begin{array}{cc}
                    \varphi_0(\lambda) & g_{i,0}(\lambda)\\
                    \varphi_1(\lambda)& g_{i,1}(\lambda)\\
                   \end{array}
                 \right|.
\end{equation}
 Let $\lambda=\rho^2$.  It is known \cite{Lev,VM} that $\varphi(x,\lambda)$ has the expression
\begin{equation}\label{2.1}
{\varphi}(x,\lambda)=\cos\rho x+\int_0^x K(x,t)\cos \rho tdt,\quad 0\le x\le a,
\end{equation}
where the kernel $K(x,t)$ is a function of two variables, which has the first partial derivatives $K_x(x,\cdot),K_t(x,\cdot)\in L^2(0,x)$, and
\begin{equation}\label{2.3}
\omega_1=K(a,a)=h+\frac{1}{2}\int_0^a q_1(t)dt.
\end{equation}
The relation \eqref{2.1} together with \eqref{2.3} yield
\begin{equation}\label{2.5}
{\varphi}_0(\lambda)=\cos\rho a+\omega_1 \frac{\sin \rho a}{\rho}-\int_0^a K_1(t)\frac{\sin \rho t}{\rho}dt,
\end{equation}
\begin{equation}\label{2.6}
{\varphi}_1(\lambda)=-\rho\sin\rho a+\omega_1 \cos \rho a+\int_0^a K_2(t)\cos \rho tdt,
\end{equation}
where $K_1(t):=K_t(a,t)$ and  $K_2(t):=K_x(a,t)$. The set $\{K_1(t), K_2(t),\omega_1\}$ is called the Cauchy data for $q_1$ and $h$.

Substituting \eqref{2.5} and \eqref{2.6} into \eqref{2.0}, we get
\begin{align}\label{2.7}
\notag -\Delta_i(\lambda)=& \rho^{-1}g_{i,1}(\lambda)\int_0^a K_1(t)\sin \rho tdt+g_{i,0}(\lambda)\int_0^a K_2(t)\cos \rho tdt \\
&-\left[\cos\rho a+\frac{\omega_1 \sin \rho a}{\rho}\right]g_{i,1}(\lambda)+g_{i,0}(\lambda)\left[\omega_1 \cos \rho a-\rho\sin\rho a\right].
\end{align}
Introduce the Hilbert space of vector-valued functions $\mathcal{H}:=L^2(0,a)\times L^2(0,a)$ with the inner product $\langle\cdot,\cdot\rangle$ defined by
\begin{equation}\label{a11}
  \langle \mathbf{h},\mathbf{ p}\rangle=\int_0^a \overline{h_1(x)}{p_1(x)}+\overline{h_2(x)}{p_2(x)}dx,\quad \forall\mathbf{ h}:=(h_1,h_2),\mathbf{p}:=(p_1,p_2)\in \mathcal{H}.
\end{equation}
Rewrite \eqref{2.7} as
\begin{equation}\label{2.11}
-\Delta_i(\lambda)=\left\langle \mathbf{K}(\cdot),\mathbf{U}_i(\cdot,\lambda)\right\rangle-f_i(\lambda) ,\quad i=0,1,
\end{equation}
where
\begin{equation}\label{2.10}
f_i(\lambda):=\left[\!\cos\rho a+\omega_1 \frac{\sin \rho a}{\rho}\right]g_{i,1}(\lambda)-g_{i,0}(\lambda)\left[\omega_1 \cos \rho a-\rho\sin\rho a\right],
\end{equation}
\begin{equation}\label{2.9}
  \mathbf{U}_i(t,\lambda):=(U_{i,1}(t,\lambda),U_{i,2}(t,\lambda)),\quad \mathbf{K}(t):=(\overline{K_1(t)},\overline{K_2(t)}),
\end{equation}
\begin{equation}\label{2.8}
U_{i,1}(t,\lambda):=g_{i,1}(\lambda)s(t,\lambda),\quad s(t,\lambda):=\frac{\sin \rho t}{\rho},\quad U_{i,2}(t,\lambda):=g_{i,0}(\lambda) c(t,\lambda),\quad c(t,\lambda):=\cos \rho t.
\end{equation}

Recall $\{\mu_{i,k}\}_{k\ge0}=\{\lambda_{i,n}\}_{n\in I_i}$, $i=0,1$. Let $m_{n}^i$ be the multiplicity of the value $\mu_{i,n}$ in the sequence $\{\mu_{i,n}\}_{n\ge0}$. Without loss of generality, assume $\mu_{i,n}=\mu_{i,n+1}=\cdot\cdot\cdot=\mu_{i,n+m_n^i-1}$. Note that $m_{n}^i=1$ for $n\ge n_i$ for some large $n_i$. Consider the set
\begin{equation*}
  \mathcal{S}_i:=\{n\in \mathbb{N}:\mu_{i,n}\ne \mu_{i,n-1},n\ge1\}\cup \{0\} ,\quad i=0,1.
\end{equation*}
It is obvious that the sequence $\{\mu_{i,n}\}_{n\in \mathcal{S}_i}$ consists of elements of $\{\mu_{i,n}\}_{n\ge0}$ being taken only once.
Denote
\begin{equation*}
  f^{\langle \nu\rangle}(\lambda):=\frac{1}{\nu!}\frac{d^\nu f(\lambda)}{d\lambda^\nu},\quad P^{\langle \nu\rangle}(t,\lambda):=\frac{1}{\nu !}\frac{\partial^\nu P(t,\lambda)}{\partial\lambda^\nu}.
\end{equation*}
Define
\begin{equation}\label{xq}
  \mathbf{U}_{n+\nu}^i(t):=\mathbf{U}_i^{\langle \nu\rangle}(t,\mu_{i,n}), \quad n\in \mathcal{S}_i,\quad \nu=\overline{0,m_n^i-1} ,\quad i=0,1,
  \end{equation}
  and
\begin{equation}\label{xq3}
  \tau_{n+\nu}^i:=f_i^{\langle \nu \rangle} (\mu_{i,n}), \quad n\in \mathcal{S}_i,\quad \nu=\overline{0,m_n^i-1} ,\quad i=0,1.
\end{equation}
Together with \eqref{2.11}--\eqref{xq3}, we get the \emph{main equations} of Inverse Problem \ref{ip2}
\begin{equation}\label{xq2}
 \left\langle \mathbf{K}(\cdot),\mathbf{U}_{n}^i(\cdot)\right\rangle=\tau_n^i,\quad n\ge 0,\quad i=0,1.
\end{equation}

Let us formulate the uniqueness results for the solution of Inverse Problem \ref{ip2}. Firstly, using the given data $a_1,$ $a_2$, $H$ and $q_2$, we can uniquely recover the functions $\mathbf{U}_i(t,\lambda)$
($i=0,1$) with the help of \eqref{2.0}, \eqref{2.9}, and \eqref{2.8}. Then, using the given eigenvalues $\{\mu_{i,n}\}_{n\ge0,i=0,1}$, we can construct the system of functions
$\{\mathbf{U}_n^i(t)\}_{n\ge0,i=0,1}$ by \eqref{xq}.

\begin{theorem}[Uniqueness 1]\label{th2}
Assume that Inverse Problem \ref{ip2} is solvable.  Then
the solution of Inverse Problem \ref{ip2} is  unique if  and only if
  the system $\{\mathbf{U}_n^i(t)\}_{n\ge0,i=0,1}$ defined in \eqref{xq} is complete in $\mathcal{H}$.
\end{theorem}

In Theorem \ref{th2}, the condition depends on the system of functions $\{\mathbf{U}_n^i(t)\}_{n\ge0,i=0,1}$ which, visually, relies on not only the subspectra but also the data $a_1,a_2,H$ and $q_2$. In the following uniqueness theorem, the condition, visually, only  depends on the  subspectra.
Denote
\begin{equation}\label{2.26}
 c_{n+\nu}^i(t):=c^{\langle\nu\rangle}(t,\mu_{i,n})=\left.\frac{1}{\nu !}\frac{\partial^\nu\cos \rho t}{\partial\lambda^\nu}\right|_{\lambda=\mu_{i,n}}, \quad n\in \mathcal{S}_i,\quad \nu=\overline{0,m_n^i-1} ,\quad i=0,1.
\end{equation}

\begin{theorem}[Uniqueness 2]\label{th2s}
Assume that Inverse Problem \ref{ip2} is solvable.  The solution of Inverse Problem \ref{ip2} is unique  if the system $\{c_n^i(t)\}_{n\ge0,i=0,1} $ defined in \eqref{2.26} is complete in $L^2(0,2a)$.
\end{theorem}

\begin{remark}
Theorem \ref{th2} gives the necessary and sufficient condition for the uniqueness of the solution of Inverse Problem \ref{ip2}. Theorem \ref{th2s} only gives  the sufficient condition.
From Lemmas \ref{l2} and \ref{l3} in Section 4, we know that
the condition in Theorem \ref{th2s} implies the condition in Theorem \ref{th2}.  However, the condition in Theorem \ref{th2s} seems easier to verify in some special cases (see the proof of Corollary \ref{cor5.1}).
In general, if $\{\mathbf{U}_n^i(t)\}_{n\ge0,i=0,1}$ is complete in $\mathcal{H}$,
 it not clear whether $\{c_n^i(t)\}_{n\ge0,i=0,1}$ is complete in $L^2(0,2a)$.
\end{remark}

\begin{remark}\label{remnew}
The eigenvalues $\{\lambda_n\}_{n\ge1}$ and coefficients $\cos \alpha_n$ and $\sin \alpha_n$ in \eqref{Horvath}, and the entire functions $f_1(\lambda)$ and $f_2(\lambda)$ in \eqref{2e} can be constructed  from the given data in Inverse Problems \ref{ip2} and \ref{ip1}, respectively. Indeed,
in \eqref{Horvath}, $\{\lambda_n\}_{n\ge1}=\{\mu_{0,k},\mu_{1,k}\}_{k\ge0}$ and $\{\alpha_n\}_{n\ge1}=\{\alpha_{0,k},\alpha_{1,k}\}_{k\ge0}$, where $\alpha_{i,k}$ ($i=0,1$) satisfy
\begin{equation*}
  \cos\alpha_{i,k}=\frac{g_{i,1}(\mu_{i,k})}{\sqrt{g_{i,1}(\mu_{i,k})^2+g_{i,0}(\mu_{i,k})^2}},\quad \sin\alpha_{i,k}=\frac{-g_{i,0}(\mu_{i,k})}{\sqrt{g_{i,1}(\mu_{i,k})^2+g_{i,0}(\mu_{i,k})^2}};
\end{equation*}
and in \eqref{2e}, \begin{equation*}
  f_1(\lambda)=g_{1,1}(\lambda),\quad f_2(\lambda)=-g_{1,0}(\lambda).
\end{equation*}
where the entire functions $g_{i,j}(\lambda)$ are defined in \eqref{2.0}. In view of this  observation, for the real-valued potential $q_1$ and simple eigenvalues, Theorem \ref{th2s} can be viewed as a variant of Theorem 1.2 in \cite{MH}. Furthermore, for Inverse Problem~\ref{ip1}, Theorem~\ref{th2s} can be obtained from
Theorem 2.1 in \cite{YBX}. However, for Inverse Problem~\ref{ip2} in the non-self-adjoint case with possible eigenvalue multiplicities, Theorem~\ref{th2s} does not follow from previously known results.
\end{remark}

\begin{remark}\label{remark2}
The above results can be easily generalized to the partial inverse problems from parts of $N+1$ subspectra, where $N\ge1$. Indeed, consider the problems $B_i=B(d_1,d_2,q_1,q_2,h,H_i ,a_1,a_2)$ $(i=\overline{0,N})$ with $H_0=\infty,H_1=H$ and $H_l\ne H_j$ for $l\ne j$. For $i = \overline{0,N}$, let $\{\mu_{i,n}\}_{n\ge0}$ be a subspectrum of the corresponding problem $B_i$. Similarly to the definitions of $\mathbf{U}_n^i(t)$ and $c_n^i(t)$ for $i=0,1$, we can also define $\mathbf{U}_n^i(t)$ and $c_n^i(t)$ for $i=\overline{2,N}$ by $\{\mu_{i,n}\}_{n\ge0,i=\overline{2,N}}$. Then we have the generalized result:  $q_1$ and $h$ are uniquely determined by $a_1$, $a_2$, $q_2$, $\omega_1$, $H_i$ and $\{\mu_{i,n}\}_{n\ge0}$,  $i=\overline{0,N}$, if and only if $\{\mathbf{U}_n^i(t)\}_{n\ge0,i=\overline{0,N}}$ is complete in $\mathcal{H}$ (or if $\{c_n^i(t)\}_{n\ge0,i=\overline{0,N}}$ is complete in $L^2(0,2a)$).
\end{remark}

The next result is the  algorithm for recovering the solution of Inverse Problem \ref{ip2}. It is known that, in the self-adjoint case, one can use the Cauchy data $\{K_1(t), K_2(t),\omega_1\}$ to recover $q_1$ and $h$ directly (see \cite{RS}). In the non-self-adjoint case, one can first use the Cauchy data to recover the functions $\varphi_0(\lambda)$ and $\varphi_1(\lambda)$ defined in \eqref{2.5} and \eqref{2.6}, and then recover the complex-valued potential $q_1$ and $h$
by the method of spectral mapping (see, e.g., \cite{B,FY}). Therefore, we only need to recover the Cauchy data.
 The solution of Inverse Problem \ref{ip2} can be found by the following algorithm under the assumption that $\{\mathbf{U}_n^{i}(t)\}_{n\ge0,i=0,1}$ defined in \eqref{xq} is a basis in $\mathcal{H}$.

\begin{algorithm}\label{alg:1}
Let $\{\mu_{i,n}\}_{n\ge0}\; (i=0,1)$ be the given subspectra  of the corresponding problems $B_i$. We have to find $q_1$ and $h$ from $a_1,$ $a_2$, $H$, $\omega_1$, and $q_2$.
\begin{enumerate}
\item Find the solutions $\psi_i(x, \lambda)$ ($i = 0, 1$) of equation \eqref{1s} with $j = 2$ under the initial conditions \eqref{initpsi} and then determine the functions $g_{i,j}$ by \eqref{2.0} from $q_2, H, a_1$, and $a_2$, where $i=0,1$ and $j=0,1$ .
\item Construct the basis $\{\mathbf{U}_n^i\}_{n\ge0}\; (i=0,1)$  by \eqref{2.9}, \eqref{2.8}, and \eqref{xq}.
\item Construct $\{\tau_n^i\}_{n\ge0}\; (i=0,1)$ by \eqref{2.10} and \eqref{xq3}.
\item Determine $\mathbf{K}=(K_1,K_2)\in \mathcal{H}$ by the following formula
\begin{equation*}
\mathbf{K}(t)= \sum_{n\ge0}\overline{\tau_n^0} \mathbf{U}_n^{0*}(t)+\sum_{n\ge0}\overline{\tau_n^1} \mathbf{U}_n^{1*}(t),
\end{equation*}
where $\{\mathbf{U}_n^{0*}(t)\}_{n\ge0}\cup\{\mathbf{U}_n^{1*}(t)\}_{n\ge0}$ is the basis, biorthonormal to  $\{\mathbf{U}_n^{0}(t)\}_{n\ge0}\cup\{\mathbf{U}_n^{1}(t)\}_{n\ge0}$.
 \item Construct the functions $\varphi_0(\lambda)$ and $\varphi_1(\lambda)$ by \eqref{2.5} and \eqref{2.6}.
\item Recover the potential $q_1$ and $h$ from  $\varphi_0(\lambda)$ and $\varphi_1(\lambda)$ by using the method of spectral mappings (see \cite{B,FY}).
\end{enumerate}
\end{algorithm}

The next result is a theorem on  the local solvability and stability of Inverse Problem \ref{ip2}. This theorem takes the error caused by the given partial potential into account.

\begin{theorem}\label{th4.1}
For $i = 0, 1$, let $\{\lambda_{i,n}\}_{n\in I_i}(=\{\mu_{i,k}\}_{k\ge0})$ be a subspectrum  of the corresponding problem $B_i=B(d_1,d_2,q_1,q_2,h,H_i ,a_1,a_2)$ with the complex-valued potentials $q_j\in L^2(0,d_j),h,H_1,$ $a_2\in \mathbb{C},a_1>0$ and $H_0=\infty$.  Suppose that the system of functions $\{\sqrt{(|\mu_{0,n}|+1)}\mathbf{U}_n^0(t)\}_{n\ge0}\cup \{\mathbf{U}_n^1(t)\}_{n\ge0}$  constructed in \eqref{xq} is a Riesz basis in $\mathcal{H}$. Then there exists $\varepsilon>0 $ (depending on the problem $B_1$) such that, for arbitrary sequences $\{\tilde{\lambda}_{i,n}\}_{n\in I_i}$ and any function $\tilde{q}_2\in L^2(0,d_2)$ satisfying
  \begin{equation}\label{4.1}
 \Lambda:= \sqrt{\sum_{n\in I_0}|{\lambda}_{0,n}-\tilde{\lambda}_{0,n}|^2+\sum_{n\in I_1}|{\lambda}_{1,n}-\tilde{\lambda}_{1,n}|^2}\le\varepsilon,
  \end{equation}
  and
  \begin{equation}\label{4.0}
\int_0^{d_2}\tilde{q}_2(t)dt=\int_0^{d_2}q_2(t)dt,\quad Q:=\|\tilde{q}_2-q_2\|_{L^2(0,d_2)}\le \varepsilon,
\end{equation}
respectively,
there exist  unique $\tilde{q}_1\in L^2(0,d_1)$ and $\tilde{h}\in \mathbb{C}$ such that  $\tilde{h}+\frac{1}{2}\int_0^{d_1}\tilde{q}_1(x)dx=\omega_1$ and, for $i = 0, 1$, $\{\tilde{\lambda}_{i,n}\}_{n \in I_i}$ is a subspectrum  of the problem $\tilde{B}_i:=B(d_1,d_2,\tilde{q}_1,\tilde{q}_2,\tilde{h},H_i ,a_1,a_2)$. Moreover,
 \begin{equation}\label{4.2}
   \|\tilde{q}_1-q_1\|_{L^2(0,d_1)} \le C (\Lambda+Q),\quad |\tilde{h}-h|\le C (\Lambda+Q),
 \end{equation}
 where $C>0$ depends only on the problem $B_1$.
\end{theorem}
\begin{remark}
From Lemma \ref{l4} in Section 4, we know that, in Theorem \ref{th4.1}, the condition that the system of functions $\{\sqrt{(|\mu_{0,n}|+1)}\mathbf{U}_n^0(t)\}_{n\ge0}\cup \{\mathbf{U}_n^1(t)\}_{n\ge0}$  constructed in \eqref{xq} is a Riesz basis in $\mathcal{H}$, can be replaced by the stronger but easier to verify condition that the system $\{c_n^i(t)\}_{n\ge0,i=0,1} $ defined in \eqref{2.26} is a Riesz basis in $L^2(0,2a)$.
\end{remark}

As a corollary, we also give the local solvability and stability result for Inverse Problem \ref{ip1}, in which the error caused by $q_2$ and $H$ is taken into account.

\begin{corollary}\label{th5.1}
Assume that $a\in(0,1/2]$. Let $\{\lambda_{1,n}\}_{n\in I}(=\{\mu_{1,k}\}_{k\ge0})$ be a subspectrum  of the problem $B_1=B(d_1,d_2,q_1,q_2,h,H ,a_1,a_2)$ with complex-valued potentials $q_j\in L^2(0,d_j),h,H,$ $a_2\in \mathbb{C}$ and $a_1>0$. Suppose that the system of functions $ \{\mathbf{U}_n^1(t)\}_{n\ge0}$  constructed in \eqref{xq} is a Riesz basis in $\mathcal{H}$.
Then, there exists $\varepsilon>0 $ (depending on the problem $B_1$) such that, for arbitrary sequence $\{\tilde{\lambda}_{1,n}\}_{n\in I}$, any  $\tilde{q}_2\in L^2(0,d_2)$ and $\tilde{H}\in \mathbb{C}$ satisfying
  \begin{equation}\label{5.1}
 \Lambda_1:= \sqrt{\sum_{n \in I}|{\lambda}_{1,n}-\tilde{\lambda}_{1,n}|^2}\le\varepsilon,
  \end{equation}
  and
  \begin{equation}\label{5.0}
\tilde{H}+\frac{1}{2}\int_0^{d_2}\tilde{q}_2(t)dt=H+\frac{1}{2}\int_0^{d_2}q_2(t)dt,\quad Q_1:=|\tilde{H}-H|+\|\tilde{q}_2-q_2\|_{L^2(0,d_2)}\le \varepsilon,
\end{equation}
respectively,
there exist  unique $\tilde{q}_1\in L^2(0,d_1)$ and $\tilde{h}\in \mathbb{C}$ such that  $\tilde{h}+\frac{1}{2}\int_0^{d_1}\tilde{q}_1(x)dx=\omega_1$, and $\{\tilde{\lambda}_{1,n}\}_{n\in I}$ is a  subspectrum  of the problem $\tilde{B}_1:=B(d_1,d_2,\tilde{q}_1,\tilde{q}_2,\tilde{h},\tilde{H} ,a_1,a_2)$. Moreover,
 \begin{equation}\label{5.2}
   \|\tilde{q}_1-q_1\|_{L^2(0,d_1)} \le C (\Lambda_1+Q_1),\quad |\tilde{h}-h|\le C (\Lambda_1+Q_1),
 \end{equation}
 where $C>0$ depends only on the problem $B_1$.
\end{corollary}

Furthermore, in the case of the half inverse problem (i.e., $d=1/2$), we do not need to require the Riesz-basis property of $ \{\mathbf{U}_n^1(t)\}_{n\ge0}$ and so obtain the following result.

\begin{corollary}\label{cor5.1}
Assume that $a=1/2$. Let $\{\lambda_{1,n}\}_{n\ge0}$ be the spectrum  of the problem $B_1=B(\frac{1}{2},\frac{1}{2},q_1,q_2,h,H ,a_1,a_2)$ with complex-valued potentials $q_j\in L^2(0,1/2),h,H,a_2\in \mathbb{C}$ and $a_1>0$. Then, there exists $\varepsilon>0 $ (depending on the problem $B_1$) such that, for an arbitrary sequence $\{\tilde{\lambda}_{1,n}\}_{n\ge0}$, any  $\tilde{q}_2\in L^2(0,1/2)$ and $\tilde{H}\in \mathbb{C}$ satisfying
  \begin{equation*}
 \Lambda_1:= \sqrt{\sum_{n\ge0}|{\lambda}_{1,n}-\tilde{\lambda}_{1,n}|^2}\le\varepsilon,
  \end{equation*}
  and
  \begin{equation*}
\tilde{H}+\frac{1}{2}\int_0^{\frac{1}{2}}\tilde{q}_2(t)dt=H+\frac{1}{2}\int_0^{\frac{1}{2}}q_2(t)dt,\quad Q_1:=|\tilde{H}-H|+\|\tilde{q}_2-q_2\|_{L^2(0,1/2)}\le \varepsilon,
\end{equation*}
respectively,
there exist unique $\tilde{q}_1\in L^2(0,\frac{1}{2})$ and $\tilde{h}\in \mathbb{C}$ such that  $\tilde{h}+\frac{1}{2}\int_0^{\frac{1}{2}}\tilde{q}_1(x)dx=\omega_1$, and $\{\tilde{\lambda}_{1,n}\}_{n\ge0}$ is the  spectrum  of the problem $\tilde{B}_1:=B(\frac{1}{2},\frac{1}{2},\tilde{q}_1,\tilde{q}_2,\tilde{h},\tilde{H} ,a_1,a_2)$. Moreover,
 \begin{equation*}
   \|\tilde{q}_1-q_1\|_{L^2(0,1/2)} \le C (\Lambda_1+Q_1),\quad |\tilde{h}-h|\le C (\Lambda_1+Q_1),
 \end{equation*}
 where $C>0$ depends only on the problem $B_1$.
\end{corollary}

\begin{remark}
Under the requirement $Q_1 = 0$, Corollaries~\ref{th5.1} and~\ref{cor5.1} can be obtained by the method of \cite{YBX}.	
\end{remark}

\section{Some asymptotic estimates}
In this section, we investigate the asymptotic behavior of the functions $\{\mathbf{U}_n^i(t)\}_{n\ge0,i=0,1}$ defined in \eqref{xq}.

 Note that $\psi_1(d_2,\lambda)$ and $\psi_1'(d_2,\lambda)$ have expressions similar to \eqref{2.5} and \eqref{2.6} with $a$ and $\omega_1$ replaced by $d_2$ and $\omega_2:=H+\frac{1}{2}\int_0^{d_2}q_2(t)dt$, respectively. Then, we obtain
 \begin{equation}\label{jg1}
  g_{1,0}(\lambda)=-a_1^{-1}\cos \rho d_2-a_1^{-1}\omega_2 \frac{\sin \rho d_2}{\rho}+\frac{1}{\rho}\int_{-d_2}^{d_2} P_{1,0}(t)e^{\mathrm{i}\rho t}dt,
 \end{equation}
  \begin{equation}\label{jg2}
  g_{1,1}(\lambda)=-a_1\rho\sin \rho d_2+(a_1\omega_2+a_2) {\cos \rho d_2}+\int_{-d_2}^{d_2} P_{1,1}(t)e^{\mathrm{i}\rho t}dt,
 \end{equation}
 where $P_{1,j}(\cdot)\in L^2(-d_2,d_2)$.
Using the transformation operator expression for $\psi_0(x,\lambda)$, we also have
 \begin{equation}\label{jg3}
  g_{0,0}(\lambda)=-a_1^{-1}\frac{\sin \rho d_2}{\rho}+a_1^{-1}\omega_0 \frac{\cos \rho d_2}{\rho^2}+\frac{1}{\rho^2}\int_{-d_2}^{d_2} P_{0,0}(t)e^{\mathrm{i}\rho t}dt,
 \end{equation}
  \begin{equation}\label{jg4}
  g_{0,1}(\lambda)=a_1\cos \rho d_2+(a_1\omega_0+a_2) \frac{\sin \rho d_2}{\rho}+\frac{1}{\rho}\int_{-d_2}^{d_2} P_{0,1}(t)e^{\mathrm{i}\rho t}dt,
 \end{equation}
where $P_{0,j}\in L^2(-d_2,d_2)$ and $\omega_0=\frac{1}{2}\int_0^{d_2}q_2(t)dt$. Substituting \eqref{2.5}--\eqref{2.6} and \eqref{jg1}--\eqref{jg4} into \eqref{2.0} and using the Paley-Wiener Theorem (see, e.g., \cite[p.101]{Yo}), we calculate
\begin{equation}\label{zop1}
 \Delta_1(\lambda)=\Delta_1^0(\lambda)+\eta_+\cos \rho+\eta_-\cos \rho (2a-1)+\int_{-1}^{1} P_{1}(t)e^{\mathrm{i}\rho t}dt,
\end{equation}
\begin{equation}\label{zop}
 \Delta_0(\lambda)=\Delta_0^0(\lambda)+\zeta_+\frac{\sin \rho}{\rho}+\zeta_-\frac{\sin \rho (2a-1)}{\rho}+\frac{1}{\rho}\int_{-1}^{1} P_{0}(t)e^{\mathrm{i}\rho t}dt,
\end{equation}
where $P_i\in L^2(-1,1)$, $i=0,1$, and
 \begin{equation*}
   \Delta_1^0(\lambda):=-\rho\left[b_+\sin \rho -b_- \sin \rho(2a-1)\right],\;\; \Delta_0^0(\lambda):=b_+\cos \rho +b_- \cos \rho(2a-1),
 \end{equation*}
\begin{equation*}
 b_\pm=\frac{a_1\pm a_1^{-1}}{2},\quad \eta_\pm=b_\pm(\omega_2\pm\omega_1)\pm\frac{a_2}{2},\quad \zeta_\pm=b_\pm(\omega_1\pm\omega_0)\pm\frac{a_2}{2}.
\end{equation*}
For $i = 0, 1$, let $\{\lambda_{i,n}^0\}_{n\ge0}$ be the zeros of $\Delta_i^0{(\lambda)}$, which are real and simple, since $a_1>0$. Denote $\rho_{i,n}^0:=\sqrt{\lambda_{i,n}^0}$. Using a similar method to Lemma 1 in \cite{AM}, it is easy to get that for each fixed $i=0,1$, the sequence $\{\rho_{i,n}^0\}$ is separated, namely, $|\rho_{i,n}^0-\rho_{i,m}^0|\ge c_0>0$ whenever $n\neq m$.
Using the standard method involving the Rouch\`{e} Theorem together with \eqref{zop1} and \eqref{zop}, one can prove the following proposition (see, e.g., \cite{FYU1,YUR1,AM}).
\begin{proposition}\label{pro1}
The eigenvalues $\{\lambda_{i,n}\}_{n\ge0}$ of the problem $B_i$ have the asymptotic behavior
\begin{equation}\label{asy1}
  \rho_{i,n}:=\sqrt{\lambda_{i,n}}=\rho_{i,n}^0+\frac{\theta_{i,n}}{\rho_{i,n}^0}+\frac{\kappa_{i,n}}{\rho_{i,n}^0},\quad \{\kappa_{i,n}\}\in l^2,\quad i=0,1,
\end{equation}
where
\begin{equation}\label{theta}
  \theta_{1,n}=\frac{\eta_+\cos \rho_{1,n}^0+\eta_-\cos \rho_{1,n}^0(2a-1)}{2\dot{\Delta}_1^0(\lambda_{1,n}^0)},\quad \theta_{0,n}=-\frac{\zeta_+\sin \rho_{0,n}^0+\zeta_-\sin \rho_{0,n}^0(2a-1)}{2\rho_{0,n}^0\dot{\Delta}_0^0(\lambda_{0,n}^0)},
\end{equation}
here $\dot{\Delta}_i^0(\lambda):=\frac{d{\Delta}_i^0(\lambda)}{d\lambda}$. Moreover, for each $i=0,1$, there is $n_i\in \mathbb{N}_0$ such that the sequence $\{\rho_{i,n}\}_{n\ge n_i}$ is separated.
\end{proposition}

Using \eqref{jg1}, \eqref{jg2} in \eqref{2.8} and \eqref{jg3}, \eqref{jg4} in  \eqref{2.10}, we obtain the following relations for large $|\rho|$:
\begin{equation}\label{jg5}
U_{1,1}(t,\lambda)=u_{1,1}(t,\lambda)+O\left(\frac{e^{|\mathrm{Im}\rho|(d_2+t)}}{|\rho|}\right),\ U_{1,2}(t,\lambda)=u_{1,2}(t,\lambda)+O\left(\frac{e^{|\mathrm{Im}\rho|(d_2+t)}}{|\rho|}\right),
\end{equation}
\begin{equation}\label{jg6}
U_{0,1}(t,\lambda)=u_{0,1}(t,\lambda)+O\left(\frac{e^{|\mathrm{Im}\rho|(d_2+t)}}{|\rho|^2}\right),\ U_{0,2}(t,\lambda)=u_{0,2}(t,\lambda)+O\left(\frac{e^{|\mathrm{Im}\rho|(d_2+t)}}{|\rho|^2}\right),
\end{equation}
where
\begin{equation}\label{jg5s}
u_{1,1}(t,\lambda)=\frac{a_1}{2}[\cos \rho (d_2+t)-\cos \rho (d_2-t)],\quad u_{1,2}(t,\lambda)=\frac{-1}{2a_1}[\cos \rho (d_2-t)+\cos \rho (d_2+t)],
\end{equation}
\begin{equation}\label{jg6s}
u_{0,1}(t,\lambda)=\frac{a_1}{2\rho}[\sin\rho (d_2+t)-\sin \rho (d_2-t)],\quad u_{0,2}(t,\lambda)=\frac{-1}{2a_1\rho}[\sin\rho (d_2-t)+\sin\rho (d_2+t)].
\end{equation}

Denote $\mathbf{u}_i(t,\lambda)=(u_{i,1}(t,\lambda),u_{i,2}(t,\lambda))$, $i=0,1$. From \eqref{jg5s} and \eqref{jg6s}, and recalling $d_2=1-a$, we have
\begin{align}\label{jg9}
\notag \langle \mathbf{u}_1,\mathbf{u}_1\rangle=&\frac{a_1^2}{4}\int_0^a [ \cos\overline{\rho}(d_2-t)-\cos \overline{\rho}(d_2+t)][ \cos{\rho}(d_2-t)-\cos {\rho}(d_2+t)]dt\\
\notag&+\frac{1}{4a_1^2}\int_0^a [ \cos\overline{\rho}(d_2-t)+\cos \overline{\rho}(d_2+t)][ \cos{\rho}(d_2-t)+\cos {\rho}(d_2+t)]dt\\
\notag=&\frac{a_1^2+a_1^{-2}}{4}\int_{1-2a}^{1} \cos(\overline{\rho}t)\cos({\rho}t)dt+\frac{a_1^{-2}-a_1^{2}}{4}\int_{1-2a}^{1} \cos({\rho}(2d_2-t))\cos(\overline{\rho}t)dt\\
\notag=&\frac{a_1^2+a_1^{-2}}{8}\int_{1-2a}^{1} [\cos(\overline{\rho}-{\rho})t+\cos(\overline{\rho}+{\rho})t]dt\\
&+\frac{a_1^{-2}-a_1^2}{8}\int_{1-2a}^{1} [\cos(2d_2{\rho}-({\rho}+\overline{\rho})t)+\cos(2d_2{\rho}-({\rho}-\overline{\rho})t)]dt
\end{align}
 \begin{align}\label{jg10}
\notag \langle \mathbf{u}_0,\mathbf{u}_0\rangle=&\frac{a_1^2}{4|\lambda|}\int_0^a [ \sin\overline{\rho}(d_2-t)-\sin\overline{\rho}(d_2+t)][ \sin{\rho}(d_2-t)-\sin {\rho}(d_2+t)]dt\\
\notag&+\frac{1}{4a_1^2|\lambda|}\int_0^a [ \sin\overline{\rho}(d_2-t)+\sin\overline{\rho}(d_2+t)][ \sin{\rho}(d_2-t)+\sin {\rho}(d_2+t)]dt\\
\notag=&\frac{a_1^2+a_1^{-2}}{4|\lambda|}\int_{1-2a}^{1} \sin(\overline{\rho}t)\sin({\rho}t)dt+\frac{a_1^{-2}-a_1^{2}}{4}\int_{1-2a}^{1} \sin(\overline{\rho}t)\sin({\rho}(2d_2-t))dt\\
\notag=&\frac{a_1^2+a_1^{-2}}{8|\lambda|}\int_{1-2a}^{1} [\cos(\overline{\rho}-{\rho})t-\cos(\overline{\rho}+{\rho})t]dt\\
&+\frac{a_1^{-2}-a_1^2}{8|\lambda|}\int_{1-2a}^{1} [\cos(2d_2{\rho}-({\rho}+\overline{\rho})t)-\cos(2d_2{\rho}-({\rho}-\overline{\rho})t)]dt
\end{align}
By Proposition \ref{pro1}, we know that $ \mathrm{Im} \sqrt{\lambda_{i,n}}\to0$ and $\mathrm{Re} \sqrt{\lambda_{i,n}}\to\infty$ as $n\to\infty$. Using \eqref{jg5} and \eqref{jg6} in \eqref{xq}, together with \eqref{jg9} and \eqref{jg10}, and with the help of the mean value theorem, we obtain
\begin{equation}\label{jg11}
  \|\mathbf{U}_{n}^1\|_{\mathcal{H}}^2=\frac{(a_1^2+a_1^{-2})a}{4}[1+o(1)]+\frac{(a_1^{-2}-a_1^{2})a}{4}\left[\cos(2\sqrt{\mu_{1,n}}(1-a))+o(1)\right],\  n\to\infty,
\end{equation}
\begin{equation}\label{jg11sa}
  \|\mathbf{U}_{n}^0\|_{\mathcal{H}}^2=\frac{(a_1^2+a_1^{-2})a}{4|\mu_{0,n}|}[1+o(1)]-\frac{(a_1^{-2}-a_1^{2})a}{4|\mu_{0,n}|}\left[\cos(2\sqrt{\mu_{0,n}}(1-a))+o(1)\right],\  n\to\infty.
\end{equation}
By virtue of the eigenvalue asymptotics, we know that $|\cos(2\sqrt{\mu_{i,n}}(1-a))|\le1$, $i=0,1$, for sufficiently large $n$.
It follows from \eqref{jg11} and \eqref{jg11sa} that
\begin{equation}\label{uni}
  \|\mathbf{U}_{n}^i\|_{\mathcal{H}}^2=\frac{(a_1^2+a_1^{-2})a}{4|\mu_{0,n}|^{1-i}}[1+c_{i,n}+o(1)],\quad i=0,1,\quad  n\to\infty,
\end{equation}
where $\{c_{i,n}\}$ is a sequence bounded by a constant less than $1$.

\section{Proofs of the uniqueness theorems}

In this section, we give the proofs of Theorems \ref{th2} and \ref{th2s}. In order to prove the necessity in Theorem~\ref{th2}, we need the following proposition on the local solvability and stability of the inverse problem by the Cauchy data.

\begin{proposition}[See \cite{XB1}]\label{thca}
Let $q_1(x)$ be a fixed complex-valued function from $L^2(0,a)$, and let $h \in \mathbb{C}$ be a fixed number. Denote by $\{K_1,K_2,\omega_1 \}$  the corresponding Cauchy data. Then there exists $\varepsilon>0$ (depending only on $q_1$ and $h $) such that, for any functions $\{\tilde{K}_1,\tilde{K}_2\}$ satisfying
\begin{equation}\label{cau}
\Xi:=\max\{\|\tilde{K}_1-K_1\|_{L^2(0,a)},\|\tilde{K}_2-K_2\|_{L^2(0,a)}\}\le \varepsilon,
\end{equation}
there exists a unique  function $\tilde{q}_1\in L^2 (0,a)$ such that $\{\tilde{K}_1,\tilde{K}_2,\omega_1 \}$ are the Cauchy data for $\tilde{q}_1$ and $\tilde{h} =\omega_1 -\frac{1}{2}\int_0^a\tilde{q}_1(x)dx$. Moreover,
\begin{equation}\label{cau1}
 \|\tilde{q}_1-q_1\|_{L^2(0,a)} \le C\Xi,\quad |\tilde{h} -h | \le C\Xi,
\end{equation}
 where $C$ depends only on $q_1$ and $h $.
\end{proposition}

\begin{proof}[Proof of Theorem \ref{th2}] Firstly, note that the given data in Inverse Problem \ref{ip2} implies the unique determination of $U(t,\lambda)$, $f(\lambda)$, $\{\mathbf{U}_n^i(t)\}_{n\ge0,i=0,1}$ and $\{\tau_n^i\}_{n\ge0,i=0,1}$. Due to the condition that Inverse problem \ref{ip2} is solvable,  assume $(q_1,h)$ is the solution such that $\{\mu_{i,n}\}_{n\ge0}$  are the  subspectra of the corresponding problems $B(d_1,d_2,q_1,q_2,h,H_i ,a_1,a_2)$ ($i=0,1$),  where $H_0=\infty$ and $H_1=H$. Let $\{K_1,K_2,\omega_1\}$ be the Cauchy data  for $q_1$ and $h$.

(Sufficiency). If  the system $\{\mathbf{U}_n^i(t)\}_{n\ge0,i=0,1}$ defined in \eqref{xq} is complete in $\mathcal{H}$, then there is at most one $\mathbf{K}(t)$ satisfying the main equations \eqref{xq2}. It is obvious that $\mathbf{K}(t) =(\overline{K_1(t)},\overline{K_2(t)})$. Note that there is at most one pair $(q_1,h)$ corresponding to the Cauchy data  $\{K_1,K_2,\omega_1\}$. Hence, the sufficiency is valid.

(Necessity). If the system $\{\mathbf{U}_n^i(t)\}_{n\ge0,i=0,1}$ defined in \eqref{xq} is not complete in $\mathcal{H}$, then there exists a
solution $\hat{\mathbf{K}}(t)=\left(\overline{\hat{K}_1(t)},\overline{\hat{K}_2(t)}\right)$ ($\|\hat{\mathbf{K}}\|_{\mathcal{H}}\neq0$) satisfying
\begin{equation}\label{xyb}
  \left\langle\hat{\mathbf{K}},\mathbf{U}_{n}^i\right\rangle=0,\quad n\ge0,\quad i=0,1.
\end{equation}
Due to the linearity of  \eqref{xyb} for $\hat{\mathbf{K}}$, we can choose $\hat{\mathbf{K}}$ such that $\|\hat{\mathbf{K}}\|_{\mathcal{H}}\le \varepsilon$ for $\varepsilon$ from Proposition \ref{thca}.  Define
\begin{equation*}
\tilde{\mathbf{K}}:=\left(\overline{\tilde{K}_1},\overline{\tilde{K}_2}\right),\quad   \tilde{K}_1:=K_1+\hat{K}_1,\quad  \tilde{K}_2:=K_2+\hat{K}_2.
\end{equation*}
Then using Proposition \ref{thca}, we know that there is a unique $\tilde{q}_1\in L^2(0,1)$ such that  $\{\tilde{K}_1,\tilde{K}_2,\omega_1 \}$  are the Cauchy data for $\tilde{q}_1$ and $\tilde{h} =\omega_1 -\frac{1}{2}\int_0^a\tilde{q}_1(x)dx$. Define the functions
\begin{equation*}
\tilde{\varphi}_0(\lambda):=\cos\rho a+ {\omega}_1 \frac{\sin \rho a}{\rho}-\int_0^a \tilde{K}_1(t)\frac{\sin \rho t}{\rho}dt,
\end{equation*}
\begin{equation*}
\tilde{\varphi}_1(\lambda):=-\rho\sin\rho a+{\omega}_1 \cos \rho a+\int_0^a \tilde{K}_2(t)\cos \rho tdt,
\end{equation*}
\begin{equation}\label{xyb2}
\tilde{\Delta}_i(\lambda)=\left\langle \tilde{\mathbf{K}}(\cdot),{\mathbf{U}}_i(\cdot,\lambda)\right\rangle-{f}_i(\lambda) ,\quad i=0,1.
\end{equation}
It is easy to get that functions $\tilde{\Delta}_i(\lambda)$ defined in \eqref{xyb2} have the expressions $\tilde{\Delta}_i(\lambda)= \tilde{\varphi}_0(\lambda)g_{i,1}(\lambda)-\tilde{\varphi}_1(\lambda)g_{i,0}(\lambda)$, $i=0,1$.
Thus, $\tilde{\Delta}_i(\lambda)$ are the characteristic functions of the corresponding problems $B(d_1,d_2,\tilde{q}_1,q_2,\tilde{h},H_i ,a_1,a_2)$, $i=0,1$.
Due to \eqref{xq2}, \eqref{xyb} and \eqref{xyb2}, we get that $\mu_{i,n}$ ($n\in \mathcal{S}_i$) are zeros of  $\tilde{\Delta}_i(\lambda)$,  $i=0,1$,  with the corresponding multiplicities $m_n^i$. Thus, $\{\mu_{i,n}\}_{n\ge0}$ are   the subspectra of the problems $B(d_1,d_2,\tilde{q}_1,q_2,\tilde{h},H_i ,a_1,a_2)$, respectively, $i=0,1$. We have proved that the solution of Inverse Problem \ref{ip2} is not unique if $\{\mathbf{U}_n^i(t)\}_{n\ge0,i=0,1}$ defined in \eqref{xq} is not complete in $\mathcal{H}$. The proof of necessity is complete.
\end{proof}

Now, let us prove Theorem \ref{th2s} by proving two lemmas.
Similar  to \eqref{2.8} and \eqref{xq}, define
\begin{equation}\label{2.13}
\mathbf{V}(t,\lambda):=(V_1(t,\lambda),V_{2}(t,\lambda)),\quad V_1(t,\lambda):=\varphi_1(\lambda)s(t,\lambda),\quad V_{2}(t,\lambda):=\varphi_0(\lambda) c(t,\lambda),
\end{equation}
and
 \begin{equation}\label{xq1}
 \mathbf{V}_{n+\nu}^i(t):=V^{\langle \nu\rangle}(t,\mu_{i,n}), \quad n\in \mathcal{S}_i,\quad \nu=\overline{0,m_n^i-1} ,\quad i=0,1.
  \end{equation}

  \begin{lemma}\label{l2}
  The system $\{\mathbf{U}_n^i(t)\}_{n\ge0,i=0,1}$ is complete in $\mathcal{H}$ if and only if $\{\mathbf{V}_n^i(t)\}_{n\ge0,i=0,1}$ is complete in $\mathcal{H}$.
  \end{lemma}
  \begin{proof}
 Let $\mathbf{{h}}=(\overline{h_1},\overline{h_2})\in \mathcal{H}$. It is sufficient to show that
\begin{equation}\label{2.20}
\langle \mathbf{h}, \mathbf{U}_n^i\rangle=0,\quad n\ge 0,\quad i=0,1
\end{equation}
  is equivalent to
  \begin{equation}\label{2.21}
\langle \mathbf{h}, \mathbf{V}_n^i\rangle=0,\quad n\ge 0,\quad i=0,1.
\end{equation}
   Consider the functions
  \begin{equation}\label{2.1g}
    G_i(\lambda):=\left\langle \mathbf{h}(\cdot),\mathbf{U}_i(\cdot,\lambda)\right\rangle=\int_0^a \bigl( h_1(t)g_{i,1}(\lambda)s(t,\lambda)+h_2(t)g_{i,0}(\lambda)c(t,\lambda) \bigr) dt  ,\quad i=0,1,
  \end{equation}
   \begin{equation}\label{2.17}
    F(\lambda):=\left\langle \mathbf{h}(\cdot),\mathbf{V}(\cdot,\lambda)\right\rangle=\int_0^a \bigl( h_1(t)\varphi_{1}(\lambda)s(t,\lambda)+h_2(t)\varphi_{0}(\lambda)c(t,\lambda)\bigr) dt.
  \end{equation}
   By the definitions of $\mathbf{U}_n^i(t)$ and $\mathbf{V}_n^i(t)$, we know that \eqref{2.20} and \eqref{2.21} are equivalent to
  \begin{equation}\label{2.14}
 G_i^{\langle\nu\rangle}(\mu_{i,n})=0,\quad n\in \mathcal{S}_i,\quad \nu=\overline{0,m_n^i-1} ,\quad i=0,1,
\end{equation}
and
\begin{equation}\label{2.15}
F^{\langle\nu\rangle}(\mu_{i,n})=0,\quad n\in \mathcal{S}_i,\quad \nu=\overline{0,m_n^i-1} ,\quad i=0,1,
\end{equation}
respectively.
Therefore, it is sufficient to show that \eqref{2.14} is equivalent to \eqref{2.15}. Let us first prove that \eqref{2.14} implies \eqref{2.15}. Note that $\varphi_0(\lambda)$ and $\varphi_1(\lambda)$ have no common zero, and so do $g_{i,0}(\lambda)$ and $g_{i,1}(\lambda)$, $i=0,1$. From \eqref{2.0} and Proposition \ref{A1}, we have
  \begin{equation}\label{2.12}
 \varphi_j^{\langle \nu\rangle}(\mu_{i,n})=\sum_{k=0}^\nu M_{n,k}^ig_{i,j}^{\langle\nu-k\rangle}(\mu_{i,n}),\quad n\in \mathcal{S}_i,\quad \nu=\overline{0,m_n^i-1},\quad j=0,1,
\end{equation}
where $M_{n,\nu}^i$ are constants.
Using \eqref{2.1g} and \eqref{2.12}, we obtain
\begin{equation*}
  \begin{split}
 &F^{\langle \nu\rangle} (\mu_{i,n})\!=\!\sum_{j=0}^\nu \int_0^a \left[h_1(t)\varphi_1^{\langle j\rangle}(\mu_{i,n})s^{\langle\nu-j\rangle}(t,\mu_{i,n})+h_2(t)\varphi_0^{\langle j\rangle}(\mu_{i,n})c^{\langle\nu-j\rangle}(t,\mu_{i,n})\right]dt\\
 &=\!\sum_{j=0}^\nu \!\sum_{l=0}^j \!M_{n,l}^i\! \int_0^a \left[h_1(t)g_{i,1}^{\langle j-l\rangle}(\mu_{i,n})s^{\langle\nu-j\rangle}(t,\mu_{i,n})+h_2(t)g_{i,0}^{\langle j-l\rangle}(\mu_{i,n})c^{\langle\nu-j\rangle}(t,\mu_{i,n})\right]dt\\
 &=\!\sum_{l=0}^\nu \!M_{n,l}^i\!\sum_{j=0}^{\nu-l} \! \int_0^a \left[h_1(t)g_{i,1}^{\langle j\rangle}(\mu_{i,n})s^{\langle\nu-l-j\rangle}(t,\mu_{i,n})+h_2(t)g_{i,0}^{\langle j\rangle}(\mu_{i,n})c^{\langle\nu-l-j\rangle}(t,\mu_{i,n})\right]dt\\
 &=\!\sum_{l=0}^\nu \!M_{n,l}^i  G_i^{\langle\nu-l\rangle}(\mu_{i,n})=0.
  \end{split}
\end{equation*}
Similarly, one can also prove that \eqref{2.15} implies \eqref{2.14}.
  \end{proof}

\begin{lemma}\label{l3}
If the system $\{c_n^i(t)\}_{n\ge0,i=0,1} $ is complete in $L^2(0,2a)$, then $\{\mathbf{V}_n^i(t)\}_{n\ge0,i=0,1}$ is complete in $\mathcal{H}$.
\end{lemma}

\begin{proof}
Substituting \eqref{2.5} and \eqref{2.6} into \eqref{2.13}, we get
\begin{equation}\label{2.22}
V_1(t,\lambda)=v_1(t,\lambda)+O\left(\frac{e^{2a|{\rm Im}\rho|}}{|\rho|}\right),\quad  V_{2}(t,\lambda)=v_2(t,\lambda)+O\left(\frac{e^{2a|{\rm Im}\rho|}}{|\rho|}\right),\quad |\rho|\to\infty,
\end{equation}
where
\begin{equation}\label{2.23}
v_1(t,\lambda)=  \frac{1}{2}[\cos \rho (a+t)-\cos \rho (a-t)],\quad v_2(t,\lambda)=  \frac{1}{2}[\cos \rho (a-t)+\cos \rho (a+t)].
\end{equation}
Substituting \eqref{2.5} and \eqref{2.6} into \eqref{2.17} and taking \eqref{2.22} and \eqref{2.23}  into account, we obtain
\begin{equation}\label{2.24}
F(\lambda)=\int_0^{2a} b_0(t)\cos \rho tdt+F_1(\rho),
\end{equation}
where
\begin{equation*}
b_0(t)=\left\{\begin{split}
&\frac{h_1(a-t)+h_2(a-t)}{2}, \quad  0<t<a, \\
&\frac{h_2(t-a)-h_1(t-a)}{2}, \quad a<t<2a,
\end{split}\right.
\end{equation*}
and the function $F_1(\rho)$ is an even and entire function of exponential type $\le 2a$, moreover,
\begin{equation*}
F_1(\rho)=O(|\rho|^{-1}),\quad |\rho|\to\infty,\quad \rho\in \mathbb{R}.
\end{equation*}
It follows from the Paley-Wiener Theorem  and the evenness of $F_1(\rho)$ that there exists $b_1(t)$ in $L^2(0,2a)$ such that
\begin{equation}\label{2.25}
F(\lambda)=\int_0^{2a} [b_0(t)+b_1(t)]\cos \rho tdt.
\end{equation}
	
Let $\mathbf{h}=(\overline{h_1},\overline{h_2})\in \mathcal{H}$ such that \eqref{2.21} holds, or equivalently, \eqref{2.15} holds. From \eqref{2.25}, \eqref{2.26} and the completeness of $\{c_n^i(t)\}_{n\ge0,i=0,1} $, we get that $b_0(t)+b_1(t)=0$ in $L^2(0,2a)$. Therefore, $F(\lambda)\equiv0$. Using \eqref{2.17} and Proposition \ref{A2} in Appendix, we conclude that $h_2(t)=0$ and $h_1(t)=0$ in $L^2(0,a)$.
\end{proof}

\begin{proof}[Proof of Theorem \ref{th2s}]
Using Lemmas  \ref{l2} and \ref{l3}, together with Theorem \ref{th2}, we immediately finish the proof of  Theorem \ref{th2s}.
\end{proof}

In the end of this section, let us prove the following lemma, which indicates that the Riesz-basicity $\mathcal{H}$ for the functional sequence $\{\sqrt{(|\mu_{0,n}|+1)}\mathbf{U}_n^0(t)\}_{n\ge0}\cup \{\mathbf{U}_n^1(t)\}_{n\ge0}$, constructed in \eqref{xq}, is reasonable. This lemma is also useful in the proof of Corollary \ref{cor5.1}.

\begin{lemma}\label{l4}
If the system $\{c_n^i(t)\}_{n\ge0,i=0,1}$ is a Riesz basis in $L^2(0,2a)$ then: (i) $\{\mathbf{V}_n^i(t)\}_{n\ge0,i=0,1}$ is a Riesz basis in $\mathcal{H}$; (ii) $\{\sqrt{(|\mu_{0,n}|+1)}\mathbf{U}_n^0(t)\}_{n\ge0}\cup \{\mathbf{U}_n^1(t)\}_{n\ge0}$ is a Riesz basis in $\mathcal{H}$.
\end{lemma}
\begin{proof}
(i) Note that $m_{n}^i=1$ for $n\ge n_i$ for some large $n_i$. Let $\{\alpha_n\}_{n\ge 0}$ be  the numbers such that $\alpha_k\ne\alpha_l $ and $\alpha_k\ne\overline{\alpha_l} $ for all $k\ne l$ and
$$\{\alpha_n\}_{n\ge n_0+n_1}=\{\sqrt{\mu_{0,n}}\}_{n\ge n_0}\cup\{\sqrt{\mu_{1,n}}\}_{n\ge n_1}.$$
Since  $\{c_n^i(t)\}_{n\ge0,i=0,1}$  is a Riesz basis in $L^2(0,2a)$, then $\{\cos \alpha_nt\}_{n\ge0}$ is also a Riesz basis in $L^2(0,2a)$. From Proposition \ref{A3} in Appendix, we have that $\{(v_1(t,\alpha_n^2),v_2(t,\alpha_n^2))\}_{n\ge0}$ is a Riesz basis in $\mathcal{H}$, where $v_j(t,\lambda) \ (j=1,2)$ are defined in \eqref{2.23}. By \eqref{2.22} and \eqref{xq1}, we know that $\{\mathbf{V}_n^i(t)\}_{n\ge0,i=0,1}$ is quadratically close to $\{(v_1(t,\alpha_n^2),v_2(t,\alpha_n^2))\}_{n\ge0}$. Using Proposition 1.8.5 in \cite{FYU1}, together with Lemma \ref{l3}, we conclude that $\{\mathbf{V}_n^i(t)\}_{n\ge0,i=0,1}$ is a Riesz basis in $\mathcal{H}$.

(ii) Using Proposition \ref{A1}, together with the definitions of $\mathbf{U}_n^i$ and $\mathbf{V}_n^i$, we obtain that $\mathbf{U}_n^i(t)=C_{n}^i\mathbf{V}_n^i(t)$ for $n\ge n_i$, where $C_{n}^i$ are nonzero constants. Note that $\{\mathbf{U}_n^i(t)\}_{n\ge0,i=0,1}$ is complete in $\mathcal{H}$ by (i) and Lemma \ref{l2}. In view of \eqref{uni}, we get the conclusion (ii).
\end{proof}

\section{Proofs of the local solvability and stability}
In this section, we prove Theorem \ref{th4.1} as well as Corollaries \ref{th5.1} and \ref{cor5.1}. For $i = 0, 1$, let  $\{\mu_{i,n}\}_{n\ge0}$ be a fixed subspectrum  of the problem $B_i$.  We shall use the data $\{\tilde{\mu}_{i,n}\}_{n\ge0}$, $\tilde{q}_2$, $\tilde{H}_1$, $a_1$, $a_2$ and $\omega_1$ to construct $\tilde{q}_1$ and $\tilde{h}$ by Algorithm~\ref{alg:1}. We agree that, if a certain
symbol $\delta_i$ denotes an object related to the problem $B_i$, then
$\tilde{\delta}_i$ will denote an analogous object related to the sequence $\{\tilde{\mu}_{i,n}\}_{n\ge0}$,  $\tilde{q}_2$ and $\tilde{H}_1$. In Theorem  \ref{th4.1}, we assume $H_1=\tilde{H}_1$.
The notation $C$ may stand for different positive constants.

Let $\tilde{\psi}_i(x,\lambda)$ be the solution of the equation $-y''+\tilde{q}_2(x)y=\lambda y$ under the initial conditions
$$ \tilde{\psi}_0(0,\lambda)=0,\; \tilde{\psi}_0'(0,\lambda)=1,\quad \tilde{\psi}_1(0,\lambda)=1,\;\tilde{\psi}_1'(0,\lambda)=\tilde{H}_1.$$
\begin{lemma}\label{l4.1}
(i) If \eqref{4.0} is fulfilled for some $\varepsilon \le 1$, then there exist functions $W_0^i\in L^2(0,d_2)\ (i=0,1)$ such that
\begin{equation}\label{4.3}
\tilde{\psi}_0(d_2,\lambda)-{\psi}_0(d_2,\lambda)=\frac{1}{\rho^2}\int_0^{d_2}W_0^0(t)\cos \rho tdt,
\end{equation}
\begin{equation}\label{4.3s}
\tilde{\psi}_0'(d_2,\lambda)-{\psi}_0'(d_2,\lambda)=\frac{1}{\rho}\int_0^{d_2}W_0^1(t)\sin\rho tdt.
\end{equation}
Moreover,  $\|W_0^i\|_{L^2(0,d_2)}\le CQ$, $i=0,1$, where $Q$ is defined in \eqref{4.0} and the constant $C>0$ depends only on $\|q_2\|_{L^2(0,d_2)}$.

(ii) If \eqref{5.0} is fulfilled for some $\varepsilon \le 1$, then there exist functions  $W_1^i\in L^2(0,d_2)\ (i=0,1)$ such that
\begin{equation}\label{4.3h}
\tilde{\psi}_1(d_2,\lambda)-{\psi}_1(d_2,\lambda)=\frac{1}{\rho}\int_0^{d_2}W_1^0(t)\sin \rho tdt,
\end{equation}
\begin{equation}\label{4.3hs}
\tilde{\psi}_1'(d_2,\lambda)-{\psi}_1'(d_2,\lambda)=\int_0^{d_2}W_1^1(t)\cos\rho tdt.
\end{equation}
Moreover, $\|W_1^i\|_{L^2(0,d_2)}\le CQ_1$, $i=0,1$, where $Q_1$ is defined in \eqref{5.0} and the constant $C>0$ depends only on $\|q_2\|_{L^2(0,d_2)}$ and $|H_1|$.
\end{lemma}

\begin{proof}
According to the theory of transformation operators \cite{Lev,XMY}, there are functions $M^0(x,t)$ and $M^1(x,t)$, having the first partial derivatives, such that
\begin{equation}\label{4.5}
  \tilde{\psi}_0(x,\lambda)={\psi}_0(x,\lambda)+\int_0^x M^0(x,t){\psi}_0(t,\lambda)dt,\quad {\psi}_0(x,\lambda)=\frac{\sin \rho x}{\rho}+\int_0^x M^1(x,t)\frac{\sin \rho t}{\rho}dt.
\end{equation}
Moreover, $M^i(x,t)$ $(i=0,1)$ satisfy that $M_{x}^i(x,\cdot),M_{t}^i(x,\cdot)\in L^2(0,x)$ for each fixed $x\in[0,d_2]$ and
\begin{equation}\label{jm1}
M^0(x,x)=\frac{1}{2}\int_0^x [\tilde{q}_2(t)-q_2(t)]dt,\ M^1(x,x)=\frac{1}{2}\int_0^x q_2(t)dt,\ M^i(x,0)=0,
\end{equation}
 Substituting the second equation in \eqref{4.5} into the first one, integrating \eqref{4.5} by parts, and taking \eqref{jm1} and the first equation in \eqref{4.0} into account, we obtain
\begin{align}\label{jm2}
\notag \tilde{\psi}_0(d_2,\lambda)-{\psi}_0(d_2,\lambda)=&\frac{1}{\rho^2}\int_0^{d_2}M^0(d_2,x)\left[-M^1(x,x)\cos \rho x+\int_0^x M_t^1(x,t)\cos \rho tdt\right]dx\\
\notag&+\frac{1}{\rho^2}\int_0^{d_2}M_t^0(d_2,t)\cos \rho tdt\\
=&\frac{1}{\rho^2}\int_0^{d_2}W_0^0(t)\cos \rho tdt,
\end{align}
with
\begin{equation}\label{jm3}
 W_0^0(t)=M_t^0(d_2,t)-M^0(d_2,t)M^1(t,t)+\int_t^{d_2} M^0(d_2,s) M_t^1(s,t)ds.
\end{equation}
Using the estimates for $M^i(x,t)$ (see, e.g., Proposition 2.1 in \cite{XMY}), we have that, for each fixed $x\in[0,d_2]$, there hold
\begin{equation}\label{4.6}
  \max_{x\ge t\ge 0}|M^0(x,t)|\le C  Q,\; \|M_x^0(x,\cdot)\|_{L^2(0,x)}\le C  Q,\; \|M_t^0(x,\cdot)\|_{L^2(0,x)}\le C  Q,
\end{equation}
where the positive constant $C$  depends only on the sum of $\|q_2\|_{L^2(0,d_2)}$ and $\|\tilde{q}_2\|_{L^2(0,d_2)}$. Note that
$\|\tilde{q}_2\|_{L^2(0,d_2)}\le \|q_2\|_{L^2(0,d_2)}+Q$ and, by \eqref{4.0}, $Q \le \varepsilon \le 1$. Therefore, we have that the constant $C$ depends only on  $\|q_2\|_{L^2(0,d_2)}$.
Using \eqref{4.6} in \eqref{jm3}, we get that $\|W_0^0\|_{{L^2(0,d_2)}}\le C Q$.
Similarly, using \eqref{4.5}  and  \eqref{4.6}, we arrive at \eqref{4.3s} with
\begin{equation}\label{jm9}
 W_0^1(t)=M_x^0(d_2,t)+\int_t^{d_2} M_x^0(d_2,s) M^1(s,t)ds,\quad \|W_0^1\|_{{L^2(0,d_2)}}\le C Q.
\end{equation}

 Note that
there are also functions $N^0(x,t)$ and $N^1(x,t)$, having the first partial derivatives, such that
\begin{equation}\label{4.5s}
  \tilde{\psi}_1(x,\lambda)={\psi}_1(x,\lambda)+\int_0^x N^0(x,t){\psi}_1(t,\lambda)dt,\quad {\psi}_1(x,\lambda)=\cos \rho x+\int_0^x N^1(x,t)\cos \rho tdt.
\end{equation}
Moreover, $N^i(x,t)$ $(i=0,1)$ satisfy $N_{x}^i(x,\cdot),N_{t}^i(x,\cdot)\in L^2(0,x)$ for each fixed $x\in[0,d_2]$ and
\begin{equation}\label{jm1s}
N^0(x,x)=\tilde{H}_1-H_1+\frac{1}{2}\int_0^x [\tilde{q}_2(t)-q_2(t)]dt,\ N^1(x,x)=H_1+\frac{1}{2}\int_0^x q_2(t)dt,
\end{equation}
Similarly, using the estimates for $N^0(x,t)$ from Proposition 2.1 in \cite{GMXA}, we get that, for each fixed $x\in[0,d_2]$, there hold
 \begin{equation}\label{4.6s}
  \max_{x\ge t\ge 0}|N^0(x,t)|\le C  Q_1,\; \|N_x^0(x,\cdot)\|_{L^2(0,x)}\le C  Q_1,\; \|N_t^0(x,\cdot)\|_{L^2(0,x)}\le C  Q_1,
\end{equation}
where the positive constant $C$  depends only on the sum of $\|q_2\|_{L^2(0,d_2)}$ and $|H_1|$, since $Q_1 \le \varepsilon \le 1$ by \eqref{5.0}.
Using \eqref{4.5s}, integrating by parts, and taking \eqref{jm1s}, \eqref{4.6s} and the first equation in \eqref{5.0}  into account, we can get \eqref{4.3h} and \eqref{4.3hs} with 
\begin{equation}\label{jm11}
 W_1^0(t)=N^0(d_2,t)N^1(t,t)-N_t^0(d_2,t)-\int_t^{d_2} N^0(d_2,s) N_t^1(s,t)ds,\quad\|W_1^0\|_{{L^2(0,d_2)}}\le C Q_1,
\end{equation}
and
\begin{equation}\label{jm13}
 W_1^1(t)=N_x^0(d_2,t)+\int_t^{d_2} N_x^0(d_2,s) N^1(s,t)ds,\quad \|W_1^1\|_{{L^2(0,d_2)}}\le C Q_1,
\end{equation}
respectively.
The proof is complete.
\end{proof}
Denote
\begin{equation}\label{ju19}
  \tilde{g}_{i,0}(\lambda):=-a_1^{-1}\tilde{\psi}_i(d_2,\lambda),\quad   \tilde{g}_{i,1}(\lambda):=a_1\tilde{\psi}_i'(d_2,\lambda) +a_2\tilde{\psi}_i(d_2,\lambda).
\end{equation}
Similarly to \eqref{2.10}--\eqref{2.8}, define
\begin{equation}\label{4.7}
\tilde{\mathbf{U}}_i(t,\lambda):=(\tilde{U}_{i,1}(t,\lambda),\tilde{U}_{i,2}(t,\lambda)),\quad \tilde{U}_{i,1}(t,\lambda):=\tilde{g}_{i,1}(\lambda)s(t,\lambda),\quad \tilde{U}_{i,2}(t,\lambda):=\tilde{g}_{i,0}(\lambda) c(t,\lambda),
\end{equation}
\begin{equation}\label{4.8}
\tilde{f}_i(\lambda):=\left[\!\cos\rho a+\omega_1 \frac{\sin \rho a}{\rho}\right]\tilde{g}_{i,1}(\lambda)-\tilde{g}_{i,0}(\lambda)\left[\omega_1 \cos \rho a-\rho\sin\rho a\right],\quad i=0,1.
\end{equation}
\begin{lemma}\label{l4.2}
If the condition \eqref{4.0} is fulfilled for some $\varepsilon\le1$, then
\begin{equation}\label{x5.1}
\sum_{n\in \mathcal{S}_0} (|\mu_{0,n}|+1)\sum_{\nu=0}^{m_n^0}\left\|(\tilde{\mathbf{U}}_0^{\langle \nu\rangle}-{\mathbf{U}}_0^{\langle \nu\rangle})(t,\mu_{0,n})\right\|_{\mathcal{H}}^2+\sum_{n\in \mathcal{S}_1} \sum_{\nu=0}^{m_n^1}\left\|(\tilde{\mathbf{U}}_1^{\langle \nu\rangle}-{\mathbf{U}}_1^{\langle \nu\rangle})(t,\mu_{1,n})\right\|_{\mathcal{H}}^2 \le CQ^2,
\end{equation}
\begin{equation}\label{x5.2}
\sum_{n\in \mathcal{S}_0} (|\mu_{0,n}|+1)\sum_{\nu=0}^{m_n^0}\left|(\tilde{f}_0^{\langle \nu \rangle}-{f}_0^{\langle \nu \rangle})(\mu_{0,n})\right|^2+\sum_{n\in \mathcal{S}_1} \sum_{\nu=0}^{m_n^1}\left|(\tilde{f}_1^{\langle \nu \rangle}-{f}_1^{\langle \nu \rangle})(\mu_{1,n})\right|^2\le CQ^2,
\end{equation}
where $C$ depends only on the problem $B_1$.
\end{lemma}
\begin{proof}
Using Lemma \ref{l4.1} with $\tilde{H}_1=H_1$, together with definitions of ${U}_{i,j}(t,\lambda)$, $\tilde{{U}}_{i,j}(t,\lambda)$, $f_i(\lambda)$ and $\tilde{f}_i(\lambda)$, we have that for $\nu\ge0$ and $ i=0,1$ there hold
\begin{equation}\label{ju1}
  \left|\tilde{U}_{i,j}^{\langle \nu \rangle}(t,\lambda)-{U}_{i,j}^{\langle \nu \rangle}(t,\lambda)\right|\le CQ\frac{e^{|\mathrm{Im} \rho|}}{(|\rho|+1)^{\nu+2-i}},\quad j=1,2,
\end{equation}
\begin{equation}\label{ju3}
  \left|\tilde{f}_{i}^{\langle \nu \rangle}(\lambda)-{f}_{i}^{\langle \nu \rangle}(\lambda)\right|\le C\frac{e^{|\mathrm{Im} \rho|a}}{|\rho|^{\nu+1-i}}\left|\int_{-d_2}^{d_2}W_i(t)e^{\mathrm{i}\rho t}dt\right|,\ \ \|W_i\|_{L^2(-d_2,d_2)}\le CQ,
\end{equation}
where $C$ depends only on $\|q\|_{L^2(0,d_2)}$.
Note that $m_n^i=1$ for $n\ge n_i$, $i=0,1$. In particular, $|\sqrt{\mu_{i,m}}-\sqrt{\mu_{i,n}}|\ge c_0>0$ for $m,n\ge n_i$ and $m\ne n$, and $|\mathrm{Im } \sqrt{\mu_{i,n}}|\le c_1<\infty$. Using the estimate \eqref{ju1} in \eqref{2.9} and \eqref{4.7}, we obtain \eqref{x5.1}. Using \eqref{ju3}, together with Proposition \ref{pA4}, we get \eqref{x5.2}.
\end{proof}

Note that multiplicities of $\mu_{i,n}$ and  $\tilde{\mu}_{i,n}$ may be distinct. However, by virtue of \eqref{4.1}, we have that  $\mathcal{S}_i\subseteq \tilde{\mathcal{S}}_i$ for sufficiently small $\varepsilon > 0$. In particular, $\tilde{m}_n^i=1$ for $n\ge n_i$.
Denote
\begin{equation}\label{4.10}
  \tilde{\mathbf{U}}_{n+\nu}^i(t):=\tilde{\mathbf{U}}_i^{\langle \nu\rangle}(t,\tilde{\mu}_{i,n}), \quad  \tilde{\tau}_{n+\nu}^i:=\tilde{f}_i^{\langle \nu \rangle} (\tilde{\mu}_{i,n}), \quad n\in \tilde{\mathcal{S}}_i,\quad \nu=\overline{0,\tilde{m}_n^i-1} ,\quad i=0,1.
  \end{equation}
Consider the system of equations
\begin{equation}\label{4.11}
 \left\langle \tilde{\mathbf{K}}(\cdot),\tilde{\mathbf{U}}_{n}^i(\cdot)\right\rangle=\tilde{\tau}_n^i,\quad n\ge 0,\quad i=0,1,
\end{equation}
where $\tilde{\mathbf{K}}=(\tilde{K}_1,\tilde{K}_2)$ is the unknown element in $\mathcal{H}$.

For each $i=0,1$, fix $k\in[0,n_i)\cap \mathcal{S}_i$, and assume that the value $\mu_{i,k}$ with the multiplicity $m_k^i$ corresponds to the numbers $\{\tilde{\mu}_{i,n}\}_{n\in M_{k}^i}$, where $M_{k}^i:=\{k,k+1,\cdot\cdot\cdot,k+m_{k}^i-1\}$. Define $\tilde{\mathcal{S}}_k^i:=\tilde{\mathcal{S}}_i\cap M_k^i$. Then the relation \eqref{4.11} for $n\in \tilde{\mathcal{S}}_k^i$ can be rewritten as
\begin{equation}\label{4.12}
 \left\langle \tilde{\mathbf{K}}(\cdot),\tilde{\mathbf{U}}_i^{\langle \nu\rangle}(\cdot,\tilde{\mu}_{i,n})\right\rangle=\tilde{f}_i^{\langle \nu \rangle} (\tilde{\mu}_{i,n}),\quad n\in \tilde{\mathcal{S}}_k^i,\quad \nu=\overline{0,\tilde{m}_n^i-1},\quad i=0,1.
\end{equation}
For each fixed $t\in[0,a]$ and $i=0,1$, let $\tilde{{E}}_{k,i,j}(t,\lambda)$, $\tilde{F}_{k,i}(\lambda)$ be  the unique polynomials of degree at most $m_k^i-1$, respectively, interpolating $\tilde{U}_{i,j}(t,\lambda) \ (j=1,2)$ and $\tilde{f}_i(\lambda)$ and their derivatives in the usual way at the points $\{\tilde{\mu}_{i,n}\}_{n\in M_k^i}$. Namely,
\begin{equation}\label{4.1x}
\tilde{{E}}_{k,i,j}^{\langle\nu\rangle}(t,\tilde{\mu}_{i,n})=\tilde{U}_{i,j}^{\langle\nu\rangle}(t,\tilde{\mu}_{i,n}),\quad    \tilde{F}_{k,i}^{\langle\nu\rangle}(\tilde{\mu}_{i,n})=\tilde{f}_i^{\langle\nu\rangle}(\tilde{\mu}_{i,n}),\quad n\in \tilde{\mathcal{S}}_k^i,\quad  \nu=\overline{0,\tilde{m}_n^i-1}.
\end{equation}
Denote $\tilde{\mathbf{{E}}}_{k,i}(t,\lambda):=(\tilde{{E}}_{k,i,1}(t,\lambda),\tilde{{E}}_{k,i,2}(t,\lambda))$.
Eqs. (\ref{4.12}) and (\ref{4.1x}) imply that
\begin{equation}\label{4.13}
 \left\langle \tilde{\mathbf{K}}(\cdot),\tilde{\mathbf{{E}}}_{k,i}^{\langle\nu\rangle}(\cdot,\tilde{\mu}_{i,n})\right\rangle=\tilde{F}_{k,i}^{\langle \nu \rangle} (\tilde{\mu}_{i,n}),\quad n\in \tilde{\mathcal{S}}_k^i,\quad \nu=\overline{0,\tilde{m}_n^i-1},\quad i=0,1.
\end{equation}
Since $\tilde{{E}}_{k,i,j}(t,\lambda)$, $\tilde{F}_{k,i}(\lambda)$  are  the polynomials of degree at most $m_k^i-1$, we have
\begin{equation}\label{4.14}
 \left\langle \tilde{\mathbf{K}}(\cdot),\tilde{\mathbf{{E}}}_{k,i}(\cdot,\lambda)\right\rangle=\tilde{F}_{k,i} (\lambda),\quad \lambda\in \mathbb{C},\quad n\in \tilde{\mathcal{S}}_k^i,\quad \nu=\overline{0,\tilde{m}_n^i-1},\quad i=0,1.
\end{equation}
In particular, we have
\begin{equation}\label{4.15}
  \left\langle \tilde{\mathbf{K}}(\cdot),\tilde{\mathbf{{E}}}_{k,i}^{\langle\nu\rangle}(\cdot,\mu_{i,k})\right\rangle=\tilde{F}^{\langle\nu\rangle}_{k,i} (\mu_{i,k}),\quad \nu=\overline{0,{m}_k^i-1},\quad i=0,1.
\end{equation}
Define the sequences $\{\tilde{\tilde{\mathbf{U}}}_n^i\}_{n\ge0}$ and $\{\tilde{\tilde{\tau}}_{n}^i\}_{n\ge0}$ for $ i=0,1$ as follows
\begin{equation}\label{4.16}
\left\{\begin{split}
  &\tilde{\tilde{\mathbf{U}}}_{n + \nu}^i(t)=\tilde{\mathbf{E}}_{n,i}^{\langle\nu\rangle}(t,{\mu}_{i,n}),\; \tilde{\tilde{\tau}}_{n + \nu}^i=\tilde{F}_{n,i}^{\langle\nu\rangle}({\mu}_{i,n}),\;  n\in \mathcal{S}_i\cap [0,n_i), \; \nu=\overline{0,{m}_n^i-1},\\
  &\tilde{\tilde{\mathbf{U}}}_n^i(t)={\tilde{\mathbf{U}}}^i_n(t),\quad\tilde{\tilde{\tau}}_n^i={\tilde{\tau}}_n^i,\quad n\ge n_i.
\end{split}\right.
\end{equation}
Then the system (\ref{4.11}) is equivalent to
\begin{equation}\label{4.17}
 \left\langle \tilde{\mathbf{K}}(\cdot),\tilde{\tilde{\mathbf{U}}}_{n}^i(\cdot)\right\rangle=\tilde{\tilde{\tau}}_n^i,\quad n\ge 0,\quad i=0,1.
\end{equation}

\begin{lemma}\label{l4.3}
There exists $\varepsilon>0 $ such that, if \eqref{4.1} and \eqref{4.0} are fulfilled, then the following estimates hold
\begin{equation}\label{4.18}
\sqrt{  \sum_{n\ge0}\left((|\mu_{0,n}|+1)\left\|\mathbf{U}_n^0-\tilde{\tilde{\mathbf{U}}}_n^0\right\|_{\mathcal{H}}^2+\left\|
\mathbf{U}_n^1-\tilde{\tilde{\mathbf{U}}}_n^1\right\|_{\mathcal{H}}^2\right)}<C (\Lambda+Q),
\end{equation}
\begin{equation}\label{4.19}
\sqrt{  \sum_{n\ge0}\left((|\mu_{0,n}|+1)\left|\tau_n^0-\tilde{\tilde{\tau}}_n^0\right|^2+\left|\tau_n^1-\tilde{\tilde{\tau}}_n^1\right|^2\right)}<C (\Lambda+Q).
\end{equation}
\end{lemma}
\begin{proof}
From the theory of the transformation operators \cite{Lev}, we know that
\begin{equation}\label{4.20}
  |\psi_i^{\langle \nu\rangle}(x,\lambda)|\le C \frac{e^{|\mathrm{Im} \rho|x}}{(|\rho|+1)^{\nu+1-i}},\quad
  |\psi_i'^{\langle \nu\rangle}(x,\lambda)|\le C \frac{e^{|\mathrm{Im} \rho|x}}{(|\rho|+1)^{\nu-i}},\quad \nu\ge0,
\end{equation}
Using \eqref{4.20} in \eqref{2.8} and \eqref{2.10}, together with the definitions of $g_{i,j}(\lambda)$, we obtain the estimates
\begin{equation}\label{4.21}
|U_{i,j}^{\langle \nu\rangle}(t,\lambda)|\le C \frac{e^{|\mathrm{Im} \rho|}}{(|\rho|+1)^{\nu+1-i}},\quad\nu\ge0,\quad i=0,1,\quad j=1,2,
\end{equation}
\begin{equation}\label{4.22}
|f_{i}^{\langle \nu\rangle}(\lambda)|\le C \frac{e^{|\mathrm{Im} \rho|}}{(|\rho|+1)^{\nu-i}},\quad\nu\ge0,\quad i=0,1.
\end{equation}
In view of \eqref{ju1} and \eqref{ju3}, and noting $Q\le \varepsilon\le 1$, we also have
\begin{equation}\label{4.23}
|\tilde{U}_{i,j}^{\langle \nu\rangle}(t,\lambda)|\le C \frac{e^{|\mathrm{Im} \rho|}}{(|\rho|+1)^{\nu+1-i}},\quad\nu\ge0,\quad i=0,1,\quad j=1,2,
\end{equation}
\begin{equation}\label{4.24}
|\tilde{f}_{i}^{\langle \nu\rangle}(\lambda)|\le C \frac{e^{|\mathrm{Im} \rho|}}{(|\rho|+1)^{\nu-i}},\quad\nu\ge0,\quad i=0,1.
\end{equation}
Since
\begin{equation*}
 \tilde{U}_{i,j}^{\langle \nu\rangle}(t,\mu_{i,n})-\tilde{U}_{i,j}^{\langle \nu\rangle}(t,\tilde{\mu}_{i,n})= (\nu+1) \int_{\tilde{\mu}_{i,n}}^{{\mu}_{i,n}} \tilde{U}_{i,j}^{\langle \nu+1\rangle}(t,\mu)d\mu,\quad\nu\ge0,\quad i=0,1,\quad j=1,2,
\end{equation*}
it follows from \eqref{4.23} that
\begin{equation}\label{4.25}
 |\tilde{U}_{i,j}^{\langle \nu\rangle}(t,\mu_{i,n})-\tilde{U}_{i,j}^{\langle \nu\rangle}(t,\tilde{\mu}_{i,n})|\le C \frac{|{\mu_{i,n}}-{\tilde{\mu}_{i,n}}|}{(\sqrt{|\mu_{i,n}|}+1)^{\nu+2-i}},\quad\nu\ge0,\quad i=0,1,\quad j=1,2.
\end{equation}
Similarly, we get
\begin{equation}\label{4.26}
 |\tilde{f}_{i}^{\langle \nu\rangle}(\mu_{i,n})-\tilde{f}_{i}^{\langle \nu\rangle}(\tilde{\mu}_{i,n})|\le C \frac{|{\mu_{i,n}}-{\tilde{\mu}_{i,n}}|}{(\sqrt{|\mu_{i,n}|}+1)^{\nu+1-i}},\quad\nu\ge0,\quad i=0,1.
\end{equation}
Note that
\begin{equation*}
   |\tilde{U}_{i,j}^{\langle \nu\rangle}(t,\tilde{\mu}_{i,n})-{U}_{i,j}^{\langle \nu\rangle}(t,\mu_{i,n})|\le  |\tilde{U}_{i,j}^{\langle \nu\rangle}(t,\tilde{\mu}_{i,n})-\tilde{U}_{i,j}^{\langle \nu\rangle}(t,{\mu}_{i,n})|+ |\tilde{U}_{i,j}^{\langle \nu\rangle}(t,{\mu}_{i,n})-{U}_{i,j}^{\langle \nu\rangle}(t,{\mu}_{i,n})|,
\end{equation*}
\begin{equation*}
   |\tilde{f}_{i}^{\langle \nu\rangle}(\tilde{\mu}_{i,n})-{f}_{i}^{\langle \nu\rangle}(\mu_{i,n})|\le  |\tilde{f}_{i}^{\langle \nu\rangle}(\tilde{\mu}_{i,n})-\tilde{f}_{i}^{\langle \nu\rangle}({\mu}_{i,n})|+ |\tilde{f}_{i}^{\langle \nu\rangle}({\mu}_{i,n})-{f}_{i}^{\langle \nu\rangle}({\mu}_{i,n})|.
\end{equation*}
Using  \eqref{4.25}, \eqref{4.26} and Lemma \ref{l4.2}, and noting $m_n^i=1$ for $n\ge n_i$, $i=0,1$, we obtain
\begin{equation}\label{4.27}
  \sum_{n\ge n_0}(|\mu_{0,n}|+1)\left\|\mathbf{U}_n^0-\tilde{\tilde{\mathbf{U}}}_n^0\right\|_{\mathcal{H}}^2+\sum_{n\ge n_1}\left\|
\mathbf{U}_n^1-\tilde{\tilde{\mathbf{U}}}_n^1\right\|_{\mathcal{H}}^2<C (\Lambda+Q)^2,
\end{equation}
\begin{equation}\label{4.28}
 \sum_{n\ge n_0}(|\mu_{0,n}|+1)\left|\tau_n^0-\tilde{\tilde{\tau}}_n^0\right|^2+\sum_{n\ge n_1}\left|\tau_n^1-\tilde{\tilde{\tau}}_n^1\right|^2<C ( \Lambda+Q)^2.
\end{equation}

Now let us consider $n\in [0,n_i)$, $i=0,1$. Fix $i=0,1$. By the definitions of $\tilde{E}_{n,i,j}(t,\lambda)$ and $\tilde{F}_{n,i}(\lambda)$, using Proposition~\ref{lA1}, we have that for each fixed $k\in[0,n_i)\cap \mathcal{S}_i$,
\begin{equation}\label{zla0}
 \left|\tilde{E}_{k,i,j}^{\langle\nu\rangle}(t,\mu_{i,k})- \tilde{U}_{i,j}^{\langle\nu\rangle}(t,\mu_{i,k})\right|\le C\max_{n\in \tilde{\mathcal{S}}_k^i}|\tilde{\mu}_{i,n}-\mu_{i,k}|,\quad \nu=\overline{0,{m}_k^i-1},
\end{equation}
 \begin{equation}\label{zla}
 \left|\tilde{F}_{k,i}^{\langle\nu\rangle}(\mu_{i,k})-\tilde{f}_i^{\langle\nu\rangle}(\mu_{i,k})\right|\le  C\max_{n\in \tilde{\mathcal{S}}_k^i}|\tilde{\mu}_{i,n}-\mu_{i,k}|,\quad \nu=\overline{0,{m}_k^i-1},
 \end{equation}
for sufficient small $\varepsilon>0$. Note that
\begin{equation*}
  \left|\tilde{E}_{k,i,j}^{\langle\nu\rangle}(t,\mu_{i,k})- {U}_{i,j}^{\langle\nu\rangle}(t,\mu_{i,k})\right|\le \left|\tilde{E}_{k,i,j}^{\langle\nu\rangle}(t,\mu_{i,k})- \tilde{U}_{i,j}^{\langle\nu\rangle}(t,\mu_{i,k})\right|+\left|\tilde{U}_{i,j}^{\langle\nu\rangle}(t,\mu_{i,k})- {U}_{i,j}^{\langle\nu\rangle}(t,\mu_{i,k})\right|,
\end{equation*}
\begin{equation*}
  \left|\tilde{F}_{k,i}^{\langle\nu\rangle}(\mu_{i,k})- {f}_{i}^{\langle\nu\rangle}(\mu_{i,k})\right|\le \left|\tilde{F}_{k,i}^{\langle\nu\rangle}(\mu_{i,k})- \tilde{f}_{i}^{\langle\nu\rangle}(\mu_{i,k})\right|+\left|\tilde{f}_{i}^{\langle\nu\rangle}(\mu_{i,k})- {f}_{i}^{\langle\nu\rangle}(\mu_{i,k})\right|,
\end{equation*}
Using Lemma \ref{l4.2} and \eqref{zla0}, \eqref{zla}, we obtain
\begin{equation*}
 \sum_{n\in M_k^i} \left\|\tilde{\tilde{\mathbf{U}}}_n^i-\mathbf{U}_n^i\right\|_{\mathcal{H}}\le C\left(Q+\max_{n\in \tilde{\mathcal{S}}_k^i}|\tilde{\mu}_{i,n}-\mu_{i,k}|\right), \quad k\in \mathcal{S}_i\cap [0,n_i),\quad i=0,1,
\end{equation*}
\begin{equation*}
 \sum_{n\in M_k^i} |\tilde{\tilde{\tau}}_n^i-\tau_n^i|\le C\left(Q+\max_{n\in \tilde{\mathcal{S}}_k^i}|\tilde{\mu}_{i,n}-\mu_{i,k}|\right),\quad k\in \mathcal{S}_i\cap [0,n_i),\quad i=0,1.
\end{equation*}
Hence
\begin{equation}\label{4.30}
  \sum_{n=0}^{n_0-1}(|\mu_{0,n}|+1)\left\|\mathbf{U}_n^0-\tilde{\tilde{\mathbf{U}}}_n^0\right\|_{\mathcal{H}}^2+\sum_{n=0}^{n_1-1}\left\|
\mathbf{U}_n^1-\tilde{\tilde{\mathbf{U}}}_n^1\right\|_{\mathcal{H}}^2<C (\Lambda+Q)^2,
\end{equation}
\begin{equation}\label{4.31}
 \sum_{n=0}^{n_0-1}(|\mu_{0,n}|+1)\left|\tau_n^0-\tilde{\tilde{\tau}}_n^0\right|^2+\sum_{n=0}^{n_1-1}\left|\tau_n^1-\tilde{\tilde{\tau}}_n^1\right|^2<C (\Lambda+Q)^2.
\end{equation}
Using \eqref{4.27}, \eqref{4.28}, \eqref{4.30}, and \eqref{4.31}, we arrive at (\ref{4.18}) and (\ref{4.19}).
\end{proof}

\begin{proof}[Proof of Theorem \ref{th4.1}]
Using Lemma \ref{l4.3} and Proposition \ref{pA5}, we get that, for sufficiently small $\varepsilon>0$, there is a unique $\tilde{\mathbf{K}}=(\tilde{K}_1,\tilde{K}_2)\in \mathcal{H}$ satisfying \eqref{4.17} that is equivalent to \eqref{4.11}. Recall the definitions \eqref{ju19}-\eqref{4.8} of $\tilde{\mathbf{U}}_i(t,\lambda)$ and $\tilde{f}_i(\lambda)$, $i=0,1$. Define the functions  $\tilde{\Delta}_i(\lambda)$ ($i=0,1$) with $\tilde{\mathbf{K}}(t)$, $\tilde{\mathbf{U}}_i(t,\lambda)$ and $\tilde{f}_i(\lambda)$
\begin{equation}\label{4.34}
\tilde{\Delta}_i(\lambda)=-\left\langle \tilde{\mathbf{K}}(\cdot),\tilde{\mathbf{U}}_i(\cdot,\lambda)\right\rangle+\tilde{f}_i(\lambda) ,\quad i=0,1.
\end{equation}
Then Eq.\eqref{4.11} together with  \eqref{4.10} imply
\begin{equation*}
  \tilde{\Delta}_i^{\langle\nu\rangle}(\tilde{\mu}_{i,n})=0,\quad n\in \tilde{\mathcal{S}}_i,\quad \nu=\overline{0,\tilde{m}_n^i-1},\quad i=0,1.
\end{equation*}
Note that
Proposition \ref{pA5} also implies $\|\tilde{\mathbf{K}}-{\mathbf{K}}\|_{\mathcal{H}}\le C(Q+\Lambda)$. Thus, using Proposition \ref{thca}, we conclude that there exists a unique $\tilde{q}_1\in L^2(0,a)$ such that $\{\tilde{K}_1,\tilde{K}_2,\omega_1 \}$ are the Cauchy data for $\tilde{q}_1$ and $\tilde{h} =\omega_1 -\frac{1}{2}\int_0^a\tilde{q}_1(x)dx$, and the estimate \eqref{4.2} is valid.
 Define the functions $\tilde{\varphi}_0(\lambda)$ and $\tilde{\varphi}_1(\lambda)$ by the Cauchy data $\{\tilde{K}_1,\tilde{K}_2,\omega_1 \}$
\begin{equation}\label{4.32}
\tilde{\varphi}_0(\lambda):=\cos\rho a+ {\omega}_1 \frac{\sin \rho a}{\rho}-\int_0^a \tilde{K}_1(t)\frac{\sin \rho t}{\rho}dt,
\end{equation}
\begin{equation}\label{4.33}
\tilde{\varphi}_1(\lambda):=-\rho\sin\rho a+{\omega}_1 \cos \rho a+\int_0^a \tilde{K}_2(t)\cos \rho tdt,
\end{equation}
Then Eqs.\eqref{4.34}, \eqref{4.32} and \eqref{4.33}, together with \eqref{4.7} and \eqref{4.8}, imply that the functions $\tilde{\Delta}_i(\lambda)$ ($i=0,1$)  defined in \eqref{4.34} have the expressions
\begin{equation}\label{4.35}
\tilde{\Delta}_i(\lambda):=\tilde{\varphi}_0(\lambda)\tilde{g}_{i,1}(\lambda)- \tilde{\varphi}_1(\lambda)\tilde{g}_{i,0}(\lambda).
\end{equation}
The proof is complete.
\end{proof}

\begin{proof}[Proof of Corollary \ref{th5.1}]
The proof of Corollary \ref{th5.1} is similar to the proof of Theorem \ref{th4.1}. Letting $\{\mu_{0,n}\}_{n\ge0}=\emptyset$ and using only (ii) of Lemma \ref{l4.1} in the proof of Theorem \ref{th4.1}, we complete the proof of Corollary \ref{th5.1}.
\end{proof}

\begin{proof}[Proof of Corollary \ref{cor5.1}]
If $a=1/2$, then
$\lambda_{1,n}$ has the following asymptotics (cf.\eqref{asy1})
\begin{equation}\label{x5.0}
\sqrt{\lambda_{1,n}}=n\pi+\frac{{\eta_++\eta_-(-1)^n}}{n\pi{b_+}}+\frac{\kappa_{1,n}}{n},\quad \{\kappa_{1,n}\}\in l^2.
\end{equation}
Using Proposition 1 in \cite{BK} with \eqref{x5.0}, we get that $\{c^{\langle \nu \rangle}(t,\lambda_{1,n})\}_{n\in \mathcal{S}_1,\nu=\overline{0,m_n^1-1}}$ is a Riesz basis in $L^2(0,1)$, where we have assumed $\mu_{1,n}=\lambda_{1,n}$ for all $n\ge0$, and $c(t,\lambda)=\cos \rho t$. Then, using Lemma \ref{l4} and Corollary \ref{th5.1}, we get Corollary \ref{cor5.1}.
\end{proof}

\noindent {\bf Acknowledgments.}
 The author Yang was supported in part by
the Natural Science Foundation of the Jiangsu
Province of China (BK 20201303). The authors are grateful to Professor Sergey Buterin for valuable suggestions.
\\[2mm]

\begin{appendix}

\noindent {\large\textbf{Appendix}}
\\[1mm]

\setcounter{equation}{0}
\renewcommand\theequation{A.\arabic{equation}}

\setcounter{lemma}{0}
\renewcommand\thelemma{A.\arabic{lemma}}

\setcounter{proposition}{0}
\renewcommand\theproposition{A.\arabic{proposition}}

In Appendix, we provide a few auxiliary propositions. One can also find partially similar results in \cite{BB,BN1}. For convenience of readers, we summarize them in Appendix.

In Propositions~\ref{A1} and~\ref{A2}, we consider a sequence of complex numbers $\{ z_n \}_{n \ge 0}$ with finite multiplicities. Let $m_{n}$ denote the multiplicity of the value $z_{n}$ in the sequence $\{z_{n}\}_{n\ge0}$. Without loss of generality, assume that $z_{n}=z_{n+1}=\cdot\cdot\cdot=z_{n+m_n-1}$ and denote
\begin{equation*}
\mathcal{S}:=\{n\in \mathbb{N}:z_{n}\ne z_{n-1},n\ge1\}\cup \{0\}.
\end{equation*}

\begin{proposition}\label{A1}
Assume that $\{z_{n}\}_{n\ge0}$ (counted with multiplicities) are the zeros of the function $$D(z):=\phi_0(z)g_1(z)-\phi_1(z)g_0(z),$$ where the  functions $\phi_j(z)$ and $g_j(z)$ are analytic at $z_n\ ({n\ge0})$, $j=0,1$.
If the vectors $(g_0(z_n),g_1(z_n))\ne(0,0)\ne(\phi_0(z_n),\phi_1(z_n))$ for $n\in \mathcal{S}$, then
there exist constants $C_{n,\nu}$, $M_{n,\nu}$, $n\in \mathcal{S}$, $\nu=\overline{0,m_n-1}$ such that $C_{n,0}\ne0$,  $M_{n,0}\ne0$ and
\begin{equation}\label{a2}
 g_{j}^{\langle \nu\rangle}(z_{n}):=\left.\frac{1}{\nu!}\frac{d^\nu g_j(z)}{dz^\nu}\right|_{z=z_n}=\sum_{k=0}^\nu C_{n,k}\phi_j^{\langle\nu-k\rangle}(z_{n}),\quad n\in \mathcal{S},\quad \nu=\overline{0,m_n-1},\quad j=0,1.
\end{equation}
  \begin{equation}\label{a1}
 \phi_j^{\langle \nu\rangle}(z_{n})=\sum_{k=0}^\nu M_{n,k}g_{j}^{\langle\nu-k\rangle}(z_{n}),\quad n\in \mathcal{S},\quad \nu=\overline{0,m_n-1},\quad j=0,1,
\end{equation}
\end{proposition}

The proof of Proposition~\ref{A1} repeats the proof of Lemma~1 in \cite{BB}, so we omit it.

\begin{proposition}\label{A2}
Let  $\phi_0(z)$ and $\phi_1(z)$ be nontrivial entire functions and
\begin{equation}\label{a5}
\int_0^a \left( h_1(t)\phi_{1}(z)\frac{\sin \sqrt{z}t}{\sqrt{z}}+h_2(t)\phi_{0}(z)\cos \sqrt{z}t \right)dt \equiv 0, \quad h_1,h_2\in L^2(0,a).
\end{equation}
Assume that $\phi_1(z)$ has the zeros $\{z_{n}\}_{n\ge0}$ (counted with multiplicities) satisfying the asymptotics
\begin{equation*}
  \sqrt{z_n}=\frac{n\pi}{a}+\kappa_n,\quad \{\kappa_n\}\in l^2.
\end{equation*}
If $\phi_0(z_n)\ne0$ for all $n$, then  $h_1=0$ and $h_2=0$ in $L^2(0,a)$.
\end{proposition}

\begin{proof}
Define
\begin{equation*}
 c^{\langle \nu\rangle}(t,z_n)=\left.\frac{1}{\nu !}\frac{\partial^\nu\cos \sqrt{z} t}{\partial z^\nu}\right|_{z=z_{n}},\quad \nu=\overline{0,m_n-1},\quad n\in \mathcal{S}.
\end{equation*}
It follows from \eqref{a5} that
\begin{equation*}
  \sum_{k=0}^\nu \phi_0^{\langle k\rangle}(z_n)\int_0^a h_2(t)c^{\langle \nu-k\rangle}(t,z_n)dt=0,\quad \nu=\overline{0,m_n-1},\quad n\in \mathcal{S}.
\end{equation*}
From the asymptotics of $z_n$, there is only a finite number of multiple values in $\{z_n\}_{n\ge0}$. Without loss of generality, assume $m_n\ge2$ for $0\le n\le n_0-1$ for some $n_0\in \mathbb{N}$, and $m_n=1$ for $n\ge n_0$.
Since $\phi_0(z_n)\ne0$ for $n\in \mathcal{S}$, we have
\begin{equation}\label{a7}
  \int_0^a h_2(t)c(t,z_n)dt=0,\quad \quad n\in \mathcal{S},
\end{equation}
\begin{equation}\label{a8}
  \int_0^a \!h_2(t)\!\left[\!c^{\langle \nu\rangle}(t,z_n)+\sum_{k=1}^\nu \frac{\phi_0^{\langle k\rangle}(z_n)}{\phi_0(z_n)}c^{\langle \nu-k\rangle}(t,z_n)\!\right]\!dt=0,\quad \nu=\overline{1,m_n-1},\quad n\in \mathcal{S}\cap [0,n_0).
\end{equation}
It is known that $\{c^{\langle \nu\rangle}(t,z_n)\}_{n\in  \mathcal{S},\nu=\overline{0,m_n-1}}$ is a Riesz basis in $L^2(0,a)$ (see Appendix in \cite{BK}). Hence, by replacing  the finite functions $\{c^{\langle \nu\rangle}(t,z_n)\}_{n\in  \mathcal{S}\cap [0,n_0),\nu=\overline{1,m_n-1}}$ with the functions
\begin{equation}\label{a6}
c^{\langle \nu\rangle}(t,z_n)+\sum_{k=1}^\nu \frac{\phi_0^{\langle k\rangle}(z_n)}{\phi_0(z_n)}c^{\langle \nu-k\rangle}(t,z_n),\quad \nu=\overline{1,m_n-1},\quad n\in \mathcal{S}\cap [0,n_0),
\end{equation}
we know that the system of the functions in \eqref{a6} and $\{c(t,z_n)\}_{n\in \mathcal{S}}$ is also complete in $L^2(0,a)$. It follows from \eqref{a7} and \eqref{a8} that $h_2(t)=0$ in $L^2(0,a)$. Returning to \eqref{a5}, we have
\begin{equation*}
  \phi_{1}(z)\int_0^a h_1(t)\frac{\sin \sqrt{z}t}{\sqrt{z}}dt=0,\quad z\in \mathbb{C},
\end{equation*}
which implies $h_1(t)=0$ in $L^2(0,a)$ since $\phi_{1}(z)$ is a nontrivial entire function.
\end{proof}

\begin{proposition}\label{A3}
Assume that $\{\alpha_n\}_{n\ge 1}$ are complex numbers satisfying  $\alpha_k\ne\alpha_l $ and $\alpha_k\ne\overline{\alpha_l} $ for all $k\ne l$. Denote $\mathbf{v}_n(t)=(v_1(t,\alpha_n^2),v_2(t,\alpha_n^2))$, where $v_j(t,\lambda)\ (j=1,2)$ are defined in \eqref{2.23}, where $\lambda=\rho^2$.
Then the following assertions are equivalent:\\
(i) $\{\cos \alpha_n t\}_{n\ge0}$ is a Riesz basis in $L^2(0,2a)$;\\
(ii) $\{\mathbf{v}_n(t)\}_{n\ge0}$ is a Riesz basis in $\mathcal{H}:=L^2(0,a)\times L^2(0,a) $, the inner product of which is defined in \eqref{a11}.
\end{proposition}
\begin{proof}
By Theorem 9 in \cite[p.32]{Yo}, we know that a system of functions $\{f_n(t)\}_{n\ge0}$ is a Riesz basis in some Hilbert space ${\mathbb{H}}$ if and only if it is complete and satisfies the two side inequality
\begin{equation*}
  C_1\sum_{n=0}^N |\beta_n|^2\le \left\|\sum_{n=0}^N \beta_n f_n\right\|_{{\mathbb{H}}}\le  C_2\sum_{n=0}^N |\beta_n|^2,
\end{equation*}
where $\{\beta_n\}$ is an arbitrary sequence, and $N\ge 0$ is an arbitrary integer, $C_1$ and $C_2$ are some fixed constants.
In view of \eqref{2.23}, we have that
\begin{align*}
\notag \langle \mathbf{v}_j,\mathbf{v}_k\rangle=&\frac{1}{4}\int_0^a [ \cos\overline{\alpha_j}(a-t)-\cos \overline{\alpha_j}(a+t)][ \cos{\alpha_k}(a-t)-\cos {\alpha_k}(a+t)]dt\\
&+\frac{1}{4}\int_0^a [ \cos\overline{\alpha_j}(a-t)+\cos \overline{\alpha_j}(a+t)][ \cos{\alpha_k}(a-t)+\cos {\alpha_k}(a+t)]dt\\
=&\frac{1}{2}\int_0^a \cos(\overline{\alpha_j}t)\cos({\alpha_k}t)dt+\frac{1}{2}\int_a^{2a} \cos(\overline{\alpha_j}t)\cos({\alpha_k}t)dt\\
=&\frac{1}{2}\int_0^{2a} \cos(\overline{\alpha_j}t)\cos({\alpha_k}t)dt.
\end{align*}
Hence, we have
\begin{equation}\label{a12}
\left\|\sum_{n=0}^N \beta_n \mathbf{v}_n\right\|_{\mathcal{H}}^2=\sum_{j=0}^N\sum_{k=0}^N\overline{\beta_j}\beta_k\langle \mathbf{v}_j,\mathbf{v}_k\rangle=\frac{1}{2}\left\|\sum_{n=0}^N \beta_n \cos({\alpha_n}t)\right\|_{L^2(0,2a)}^2.
\end{equation}
In view of \eqref{a12}, it only remains to show that  $\{\cos \alpha_n t\}_{n\ge0}$ is complete in $L^2(0,2a)$ if and only if $\{(\mathbf{v}_n(t)\}_{n\ge0}$ is complete in $\mathcal{H} $.

 Assume that $\{\cos \alpha_n t\}_{n\ge0}$ is complete in $L^2(0,2a)$. Let $\mathbf{h}=(\overline{h_1},\overline{h_2})\in \mathcal{H}$ such that $\langle \mathbf{h},\mathbf{v}_n\rangle=0$ for all $n\ge0$. Then
the function
\begin{equation}\label{a14}
  F_0(\lambda):=\int_0^a \left( h_1(t)v_1(t,\lambda)+h_2(t)v_2(t,\lambda) \right)dt
\end{equation}
has zeros $\{\alpha_n\}_{n\ge0}$.
By the definition of $v_j(t,\lambda)$ (cf.\eqref{2.23}), we get
  \begin{equation}\label{a13}
    F_0(\lambda)=\int_0^{2a} b_0(t)\cos \rho tdt,\quad b_0(t)=\left\{\begin{split}
           &\frac{h_1(a-t)+h_2(a-t)}{2}, \quad  0<t<a, \\
              &\frac{h_2(t-a)-h_1(t-a)}{2}, \quad a<t<2a.
         \end{split}\right.
  \end{equation}
Since $\{\cos \alpha_n t\}_{n\ge0}$ is complete in $L^2(0,2a)$, we have $b_0(t)=0$ in $L^2(0,2a)$. Hence $h_1(t)=h_2(t)=0$ in $L^2(0,a)$.

Assume that $\{\mathbf{v}_n(t)\}_{n\ge0}$ is complete in $\mathcal{H} $. Let $b\in L^2(0,2a)$ such that $\int_0^{2a}b(t)\cos \alpha_n tdt=0$ for all $n\ge0$. Namely, the function
$F_2(\lambda)=\int_0^{2a}b(t)\cos \rho tdt$ has zeros $\{\alpha_n\}_{n\ge0}$. Note that the function $F_2(\lambda)$ also has the form of \eqref{a14} with
$$h_1(t)=b(a-t)-b(a+t),\quad h_2(t)=b(a-t)+b(a+t),\quad t\in(0,a).$$
Since $\{(\mathbf{v}_n(t)\}_{n\ge0}$ is complete in $\mathcal{H} $, then $h_1(t)=h_2(t)=0$ in $L^2(0,a)$ and so $b(t)=0$ in $L^2(0,2a)$.
\end{proof}

\begin{proposition}\label{pA4}
Let $\{\rho_n\}$ be separated complex numbers with bounded imaginary parts, namely, $|\rho_n-\rho_m|\ge c_0>0$ whenever $n\ne m$ and $\sup_{n} |\mathrm{Im }\rho_n|\le c_1<\infty$. If $W\in L^2(-b,b)$, then
\begin{equation}\label{akz}
 \sum_{n}\left|\int_{-b}^b W(t)e^{i \rho_n t}dt\right|^2\le C_0 \|W\|^2_{L^2(-b,b)},
\end{equation}
where $C_0>0$ depends only on $b$, $c_0$ and $c_1$.
\end{proposition}
\begin{proof}
Define $f(\rho)=\int_{-b}^b W(t)e^{i \rho t}dt$. Then $f(\rho)$ is an entire function  of exponential type $\le b$. By the theory of the Fourier transform, we know that $\|f\|_{L^2(-\infty,\infty)}^2=2\pi \|W\|_{L^2(-b,b)}^2$.
Since $\{\rho_n\}$ are  separated complex numbers with bounded imaginary parts, from Theorem 17 and its Remark in \cite[p.96-98]{Yo}, we have
\begin{equation*}
 \sum_{n}\left|f (\rho_n)\right|^2\le C_1 \|f\|_{L^2(-\infty,\infty)}^2=2\pi C_1 \|W\|_{L^2(-b,b)}^2,
\end{equation*}
which implies \eqref{akz}. Another proof of Proposition~\ref{pA4} can be obtained from Lemma~1 in \cite{ButMN}.
\end{proof}

\begin{proposition}[See, e.g., \cite{YBX}]\label{pA5}
Let $\{ v_n \}$ be a Riesz basis in a Hilbert space $\mathbb{H}$. Then there exists $\varepsilon > 0$, such that every sequence
$\{ \tilde v_n \}$, satisfying
$$
  \mathcal{V} := \left( \sum_n \| v_n - \tilde v_n \|^2_{ \mathbb H} \right)^{1/2}\le \varepsilon,
$$
is also a Riesz basis in $\mathbb{H}$. Furthermore, for some $\tau\in \mathbb{H}$, denote $\tau_n := (\tau, v_n)_\mathbb{H}$, where $(\cdot,\cdot)_{\mathbb{H}}$ is the inner product in $\mathbb{H}$,
then for any sequence $\{ \tilde \tau_n \}$ satisfying
$$
    \Omega := \left( \sum_n |\tau_n - \tilde \tau_n|^2 \right)^{1/2} \le \varepsilon,
$$
there exists a unique $\tilde \tau \in \mathbb{H}$, such that $\tilde \tau_n = (\tilde \tau, \tilde v_n)_{\mathbb{H}}$ for all $n$,
moreover,
$$
\| \tau - \tilde \tau \|_{\mathbb{H}} \le C (\mathcal{V}  + \Omega),
$$
where the constant $C$ depends only on $\{ v_n \}$ and $\tau$.
\end{proposition}

To deal with the multiple eigenvalues in the local solvability and stability, we need the following proposition.

\begin{proposition}[See \cite{MW}]\label{lA1}
  Assume that $f(z)$ is an entire function, and $z_1$,..., $z_m$ (not necessarily distinct) are in the disc $\{z\colon |z-z_0|\le r<1/2\}$. Let $p(z)$ be the unique polynomial of degree at most $m-1$ which interpolates $f(z)$ and its derivatives in the usual way at the points $z_j$, $j=\overline{1,m}$: namely, if $z_j$ appears $m_j$ times, then $p^{(n)}(z_j)=f^{(n)}(z_j)$ for $n=\overline{0,m_j-1}$. Then for each $j=\overline{0,m-1}$,
  \begin{equation}\label{xqi}
    \left|f^{(j)}(z_0)-p^{(j)}(z_0)\right|\le C r^{m-j}\sup_{|z-z_0|=1}\left|f(z)\right|,
  \end{equation}
  here the constant $C$ depends only on $m$.
\end{proposition}

\end{appendix}

\end{document}